\documentclass[11pt]{amsart}

\title[Positive flow-spines and contact $3$-manifolds]{Positive flow-spines and contact $3$-manifolds}

\author[Ishii]{Ippei Ishii}
\address{Department of Mathematics, Faculty of Science and Technology, Keio University, 3-14-1 Hiyoshi, Kohoku, Yokohama 223-8522, Japan}

\author[Ishikawa]{Masaharu Ishikawa}
\address{Department of Mathematics, Hiyoshi Campus, Keio University, 
4-1-1 Hiyoshi, Kohoku, Yokohama 223-8521, Japan}
\email{ishikawa@keio.jp}

\author[Koda]{Yuya Koda}

\address{
Department of Mathematics and International Institute for Sustainability with Knotted Chiral Meta Matter, Hiroshima University, 1-3-1 Kagamiyama, Higashi-Hiroshima 739-8526, Japan}
\email{ykoda@hiroshima-u.ac.jp}

\author[Naoe]{Hironobu Naoe}
\address{Department of Mathematics, Chuo University, 
1-13-27 Kasuga, Bunkyo-ku, Tokyo, 112-8551, Japan}
\email{naoe@math.chuo-u.ac.jp}

\usepackage{amsfonts,amsmath,amssymb,amscd}
\usepackage{amsthm}
\usepackage{latexsym}
\usepackage{graphicx}
\usepackage{psfrag}
\usepackage{color}

\usepackage{fancybox, ascmac}
\usepackage{multicol}

\theoremstyle{plain}
\newtheorem*{theorem*}{Theorem}
\newtheorem*{lemma*} {Lemma}
\newtheorem*{corollary*} {Corollary}
\newtheorem*{proposition*}{Proposition}
\newtheorem*{conjecture*}{Conjecture}
\newtheorem{theorem}{Theorem}[section]
\newtheorem{lemma}[theorem]{Lemma}

\newtheorem{proposition}[theorem]{Proposition}

\newtheorem{remark}[theorem]{Remark}

\theoremstyle{remark}

\newtheorem*{definition}{Definition}

\newtheorem{example}{Example}
\newtheorem*{example*}{Example}

\theoremstyle{definition}

\newtheoremstyle{citing}
  {}
  {}
  {\itshape}
  {}
  {\bfseries}
  {.}
  {.5em}
  {\thmnote{#3}}

\theoremstyle{citing}

\newcommand{\Integer}{\mathbb{Z}}
\newcommand{\Real}{\mathbb{R}}
\newcommand{\Complex}{\mathbb{C}}
\newcommand{\Rational}{\mathbb{Q}}

\newcommand{\Int}{\mathrm{Int\,}}

\newcommand{\Nbd}{\mathrm{Nbd}}

\newcommand{\pr}{\mathrm{pr}}

\makeatletter
\@addtoreset{equation}{section}

\makeatother

\begin{document}

\begin{abstract}
A flow-spine of a $3$-manifold is a spine admitting a flow that is
transverse to the spine, where the flow in the complement of the spine
is diffeomorphic to a constant flow in an open ball.
We say that a contact structure on a closed, connected, oriented $3$-manifold is supported by a flow-spine if it has a contact form whose Reeb flow is a flow of the flow-spine. It is known by Thurston and Winkelnkemper that any open book decomposition of a closed oriented $3$-manifold supports a contact structure. 
In this paper, we introduce a notion of positivity for flow-spines and prove that any positive flow-spine of a closed, connected, oriented $3$-manifold supports a contact structure uniquely up to isotopy. The positivity condition is critical to the existence of the unique, supported contact structure, which is also proved in the paper.
\end{abstract}

\maketitle


\section{Introduction}

A {\it spine} is a $2$-dimensional polyhedron embedded in a closed $3$-manifold obtained from the manifold by removing an open ball and collapsing it from the boundary. 
For instance, the $2$-skeleton of the dual of a one-vertex triangulation is a spine.
Spines are related to the study of non-singular flows in $3$-manifolds as follows. 
For a given non-singular flow, set a disk transverse to the flow and intersecting all orbits, and then float the boundary of the disk smoothly until it arrives in the disk itself. The object obtained in this way is a spine equipped with a special structure, called a {\it flow-spine}. A flow-spine was introduced by the first author in~\cite{Ish86} and developed further 
in~\cite{Ish92, EI05}. 
This is very useful when one studies non-singular flow combinatorially since any non-singular flow has a flow-spine~\cite{Ish86}. 
A spine is described by a decorated trivalent graph on $S^2$ called the {\it DS-diagram},
that was introduced in a paper of Ikeda and Inoue~\cite{II85}. The DS is an abbreviation of Dehn-Seifert. 
A flow-spine is described by a DS-diagram with a specified simple loop called an {\it $E$-cycle}. Flow-spines are also studied by Benedetti and Petronio in~\cite{BP97} in the context of branched standard spines. 

A contact structure is a hyperplane distribution that is non-integrable everywhere.
In this paper, all 3-manifolds are oriented, and contact structures are always assumed to be
positive, that is, a contact form $\alpha$ satisfies $\alpha\land d\alpha>0$ with respect
to the orientation on the ambient 3-manifold.
Flexibility of contact structures allows us to study them combinatorially, especially in the $3$-dimensional case. Martinet and Lutz developed Dehn surgery technique for contact $3$-manifolds~\cite{Mar71, Lut77}, and Thurston and Winkelnkemper related contact structures to open book decompositions of $3$-manifolds~\cite{TW75}. 
Tightness and fillability were introduced and studied by
Bennequin~\cite{Ben83}, Eliashberg~\cite{Eli89, Eli90, Eli92}, Gromov~\cite{Gro85}
and many other mathematicians.
Great progress had been made by Giroux. 
There are two big directions in his works on $3$-dimensional contact topology: one is convex surface theory~\cite{Gir91} and the other is the so-called Giroux correspondence~\cite{Gir02}. 
Branched surfaces are sometimes used in the study of contact structures, see for instance~\cite{BP00, CGH09}.

The aim of this paper is to relate flow-spines of $3$-manifolds to contact structures by regarding Reeb flows as flows of flow-spines. 
This idea had appeared in the paper of Benedetti and Petronio~\cite{BP00}. They focused on the characteristic foliation on a branched standard spine embedded in a contact $3$-manifold and studied the contact structure using techniques in convex surface theory of Giroux. The point is that they did not use Reeb vector fields so much since convex surface theory uses rather contact vector fields.
In this paper, we focus on Reeb flows more and define the correspondence between flow-spines and contact structures as follows:

\begin{definition}
\begin{itemize}
\item[(1)] A flow $\mathcal F$ is said to be {\it carried by a flow-spine $P$} if $\mathcal F$ is a flow of $P$.
\item[(2)] A contact structure $\xi$ on $M$ is said to be {\it supported} by a flow-spine $P$
if there exists a contact form $\alpha$ on $M$ such that $\xi=\ker\alpha$ and 
its Reeb flow is carried by $P$.
\end{itemize}
\end{definition}

To our aim, we need to introduce a notion of positivity for flow-spines. Each region of a flow-spine is canonically oriented since it is transverse to the flow. This structure is called a {\it branching}. A branched simple polyhedron has two kinds of vertices: the vertex on the left in Figure~\ref{fig3} is said to be {\it of $\ell$-type} and the one on the right is {\it of $r$-type}. We define a flow-spine to be {\it positive} if it has at least one vertex and all vertices are of $\ell$-type%
\footnote{
The positivity for flow-spines introduced in this paper is different from the one introduced in~\cite{Ish86}.}.
An essential reason for concerning this condition is that we cannot expect one of the existence and the uniqueness of a contact manifold for a given flow-spine without this condition, see Theorem~\ref{thm00}. Another reason, which is more technical, is that a $1$-form called a {\it reference $1$-form}, which plays a key role in the proof of Theorem~\ref{thm01}, 
cannot be defined unless we restrict flow-spines to positive ones. 

\begin{figure}[htbp]
\begin{center}
\includegraphics[width=7.5cm, bb=146 625 446 712]{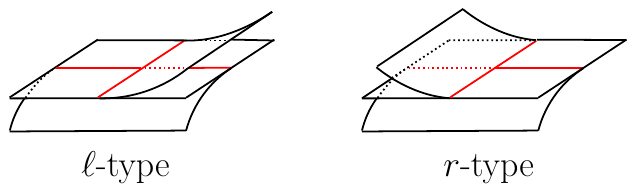}
\caption{Vertices of $\ell$-type and $r$-type. Here, the ambient space is equipped with the right-handed orientation.}\label{fig3}
\end{center}
\end{figure}

With these observations, we restrict our attention to positive flow-spines.
The main theorem of this paper is the following.


\begin{theorem}\label{thm01}
For any positive flow-spine $P$ of a closed, oriented $3$-manifold $M$, 
there exists a unique contact structure supported by $P$ up to isotopy.
\end{theorem}

Recall that two contact structures $\xi_0$ and $\xi_1$ are said to be {\it isotopic} if there exists a one-parameter family of contact forms connecting two contact forms $\alpha_0$ and $\alpha_1$ whose kernel are $\xi_0$ and $\xi_1$, respectively.
In particular, $(M,\xi_0)$ and $(M,\xi_1)$ are contactomorphic by Gray's stability~\cite{Gra59}.
The existence of a poistive flow-spine for a given contact $3$-manifold will be discussed in a forthcoming paper.

To prove Theorem~\ref{thm01}, 
we first give a reference $1$-form $\eta$ on the $3$-manifold $M$ with respect to the flow-spine $P$ explicitly and then consider a contact form on $M$ of the form $\hat\beta+R\eta$, where $\hat\beta$ is a $1$-form defined by extending a $1$-form $\beta$ on $P$ with $d\beta>0$ to the complement $M\setminus P$ and $R$ is a sufficiently large positive real number. This is analogous to the contact form $(1-t)\beta+t\phi^*\beta+Rdt$ for open books used by Thurston and Winkelnkemper and then in the Giroux correspondence, where $\phi$ is the monodromy diffeomorphism and $t\in S^1$ is the parameter of pages of the open book.  The discussion in the open book case is easier because $dt$ is closed. In the case of flow-spine, although $\eta$ is not closed, it satisfies $\eta\land d\eta\geq 0$, that is, it gives a confoliation~\cite{ET98}. We use the positivity of flow-spines to get this property. Then Theorem~\ref{thm01} follows by the same strategy as the Giroux correspondence though the argument is much more complicated. 

This paper is organized as follows: 
In Section~2, we shortly recall terminologies in $3$-dimensional contact topology
that will be used in this paper. In Section~3, we introduce flow-spines and DS-diagrams.
New observation starts from Section~4, where we introduce the admissibility condition and the notion of positivity for flow-spines. In Section~5, we introduce the reference $1$-form $\eta$ for a positive flow-spine that plays an important role in the proof of the main theorem.
The proof of the existence of the supported contact structure in Theorem~\ref{thm01}
is given in Section~6 and the proof of its uniqueness is given in Section~7.

The authors wish to express their gratitude to Riccardo Benedetti for many insightful comments.
The second author would like to thank Shin Handa and Atsushi Ichikawa for useful discussions in the early stages of research.
The second author is supported by JSPS KAKENHI Grant Numbers JP19K03499,
JP17H06128 and Keio University Academic Development Funds for Individual Research.
The third author is supported by JSPS KAKENHI Grant Numbers
JP15H03620, JP17K05254, JP17H06463, and JST CREST Grant Number JPMJCR17J4.
The fourth author is supported by JSPS KAKENHI Grant Number JP19K21019.

\section{Contact $3$-manifolds}\label{sec1}

Throughout this paper, for a polyhedral space $X$, $\Int X$ represents the interior of $X$,
$\partial X$ represents the boundary of $X$, and $\Nbd(Y;X)$ represents a closed regular neighborhood of a subspace $Y$ of $X$ in $X$,
where $X$ is equipped with the natural PL structure if $X$ is a smooth manifold.
The set $\Int\Nbd(Y;X)$ is the interior of $\Nbd(Y;X)$ in $X$.

In this section, we briefly recall notions and known results in $3$-dimensional contact topology
that will be used in this paper.
The reader may find general explanation, for instance, in~\cite{Etn06, Gei08, OS04}.

Let $M$ be a closed, oriented, smooth $3$-manifold. A {\it contact structure} on $M$ is the $2$-plane field on $M$ given by the kernel of a $1$-form $\alpha$ on $M$ satisfying $\alpha\land d\alpha\ne 0$ everywhere. The $1$-form $\alpha$ is called a {\it contact form}.
If $\alpha\land d\alpha>0$ everywhere on $M$ then the contact structure given by $\ker\alpha$ is called a {\it positive contact structure}
and the $1$-form $\alpha$ is called a {\it positive contact form}.
The pair of a closed, oriented, smooth $3$-manifold $M$ and a contact structure $\xi=\ker\alpha$ on $M$ is called a {\it contact $3$-manifold}
and denoted by $(M,\xi)$. In this paper, 
by a contact structure we mean a positive one.

There are two ways of classification of contact $3$-manifolds: up to isotopy and up to contactomorphism. Two contact structures $\xi$ and $\xi'$ on $M$ are said to be {\it isotopic} if there exists a one-parameter family of contact forms $\alpha_t$, $t\in [0,1]$, such that $
\xi=\ker\alpha_0$ and $\xi'=\ker\alpha_1$. Two contact $3$-manifolds $(M,\xi)$ and $(M',\xi')$
are said to be {\it contactomorphic} if there exists a diffeomorphism $\phi:M\to M'$ such that $\phi_*(\xi)=\xi'$. The map $\phi$ is called a {\it contactomorphism}. 
If $M=M'$ then we also say that $\xi$ and $\xi'$ are contactomorphic. The Gray theorem states that if two contact structures are isotopic then they are contactomorphic~\cite{Gra59}.

The contact structure $\xi_{\rm std}$ on $S^3$ given as the kernel of the $1$-form $\alpha_{\rm std}=\sum_{j=1}^2 (x_jdy_j-y_jdx_j)|_{S^3}$ is called the {\it standard contact structure on $S^3$}, where $(x_1+\sqrt{-1}y_1, x_2+\sqrt{-1}y_2)$ are the standard coordinates of $\Complex^2$ and $S^3$ is the unit sphere in $\Complex^2$. We also say that $\xi_{\rm std}$ is the $2$-plane field given by the complex tangency at each point of $S^3$ in $\Complex^2$. This contact structure satisfies the following important property, called ``tightness'', which was shown by Bennequin~\cite{Ben83}. A contact structure $\xi$ on $M$ is said to be {\it tight} if there does not exist a disk $D$ embedded in $M$ such that $\partial D$ is everywhere tangent to $\xi$ and the framing of $D$ along $\partial D$ coincides with that of $\xi$. Otherwise $\xi$ is said to be {\it overtwisted} and the disk $D$ is called an {\it overtwisted disk}. Note that the tightness is an invariant of contact $3$-manifolds up to contactomorphism. 

Next we introduce the Reeb vector field. Let $\alpha$ be a contact form on $M$.
A vector field $X$ on $M$ determined by the conditions $d\alpha(X,\cdot)=0$ and $\alpha(X)=1$ is called the {\it Reeb vector field} of $\alpha$ on $M$.
Such a vector field is uniquely determined by $\alpha$ and we denote it by $R_\alpha$.
The non-singular flow on a $3$-manifold $M$ generated by a Reeb vector field is called a {\it Reeb flow}.
The Reeb vector field plays important roles in many studies in contact geometry and topology.
In $3$-dimensional contact topology, it is used to give a correspondence between contact structures and open book decompositions of $3$-manifolds.
It is proved by Thurston and Winkelnkemper that for any open book there exists a supported contact structure, that is, the kernel of a contact form whose Reeb flow is transverse to the pages and tangent to the binding~\cite{TW75}. As a corollary, it follows that any closed, oriented, smooth $3$-manifold admits a contact structure. 
Moreover, it is proved by Giroux that such a contact structure is unique up to isotopy. This is called the {\it supported contact structure} and is used 
when he introduced the correspondence between contact structures and open book decompositions of $3$-manifolds~\cite{Gir02}. Note that the existence of a contact structure for any $3$-manifold was first proved by Martinet~\cite{Mar71}. Our main theorem, Theorem~\ref{thm01}, can be seen as a flow-spine version of 
these results. 

\section{Flow-spines and DS-diagrams}

\subsection{Branched polyhedron}

A compact topological space $P$ is called a {\it simple polyhedron}, or a 
{\it quasi-standard polyhedron}, 
if every point of $P$ has a regular neighborhood homeomorphic to 
one of the three local models shown in Figure~\ref{fig1}. 
A point whose regular neighborhood is shaped on the model 
(iii) is called a {\it true vertex} of $P$ (or {\it vertex} for short), and we denote the set of true vertices of $P$  by $V(P)$. 
The set of points whose regular neighborhoods are shaped on the models (ii) and (iii) is called 
the {\it singular set} of $P$, and we denote it by $S(P)$. 
Each connected component of $P\setminus S(P)$ is called a {\it region} of $P$ and each connected component of $S(P)\setminus V(P)$ is called an {\it edge} of $P$.
A simple polyhedron $P$ is said to be {\it special}, or {\it standard},
if each region of $P$ is an open disk and each edge of $P$ is an open arc.
Throughout this paper, we assume that all regions 
are orientable.

\begin{figure}[htbp]
\begin{center}
\includegraphics[width=10.0cm, bb=129 621 543 712]{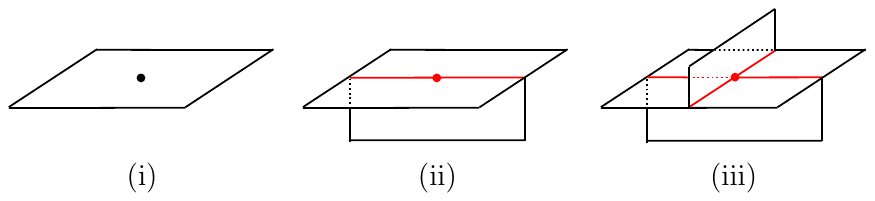}
\caption{The local models of a simple polyhedron.}\label{fig1}
\end{center}
\end{figure}

A {\it branching} of a simple polyhedron $P$ is an assignment of orientations to
regions of $P$ such that the three orientations on each edge of $P$ induced by the three adjacent regions do not agree.
We note that even though each region of a simple polyhedron $P$ is orientable, 
$P$ does not necessarily admit a branching. 
See~\cite{Ish86, BP97, Kod07, Pet12} for general properties of branched polyhedra.

\subsection{Flow-spines and DS-diagrams}\label{sec22}

A polyhedron $P$ is called a {\it spine} of a closed, connected, oriented $3$-manifold $M$ 
if it is embedded in $M$ and $M$ with removing an open ball collapses onto $P$.
If a spine is simple then it is called a {\it simple spine}.

If a spine $P$ admits a branching, then it allows us to smoothen $P$ in the ambient manifold $M$ as in the local models shown in Figure~\ref{fig2}. 
A point of $P$ whose regular neighborhood is shaped on the model~(iii) is called a {\it vertex of $\ell$-type} and that on the model~(iv) is a {\it vertex of $r$-type}.

\begin{figure}[htbp]
\begin{center}
\includegraphics[width=14.0cm, bb=129 649 587 712]{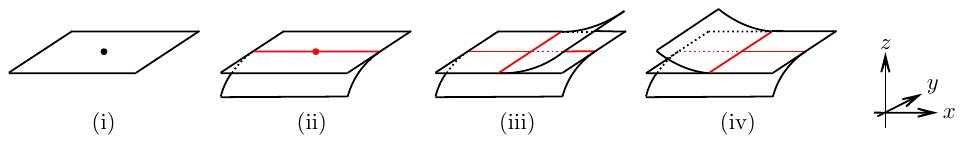}
\caption{The local models of a branched simple polyhedron.}\label{fig2}
\end{center}
\end{figure}

\begin{definition}
Let $M$ be a closed, connected, oriented $3$-manifold.
\begin{itemize}
\item[(1)] Let $(M,\mathcal F)$ be a pair of $M$ and a non-singular flow $\mathcal F$ on $M$. A simple spine $P$ of $M$ is called a {\it flow-spine} of $(M,\mathcal F)$ if, for each point $p\in P$, there exists a positive coordinate chart $(U; x,y,z)$ of $M$ around $p$ such that $(U,p)$ is one of the models in Figure~\ref{fig2}, where $\mathcal F|_U$ is generated by $\frac{\partial}{\partial z}$, and $\mathcal F|_{M\setminus P}$ is a constant vertical flow shown in Figure~\ref{fig2-2}.
\item[(2)] A branched simple spine $P$ of $M$ is called a {\it flow-spine} of $M$ if it is a flow-spine of $(M,\mathcal F)$ for some non-singular flow $\mathcal F$ on $M$.
The flow $\mathcal F$ is said to be {\it carried by $P$}.
\end{itemize}
\end{definition}


\begin{figure}[htbp]
\begin{center}
\includegraphics[width=4.5cm, bb=177 609 353 713]{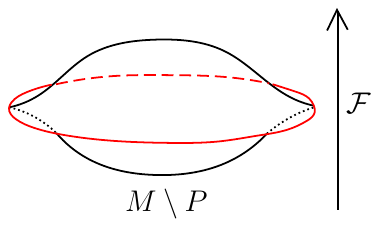}
\caption{The flow in the complement $M\setminus P$.}\label{fig2-2}
\end{center}
\end{figure}

Let $B^3$ be the unit ball in $\Real^3$ equipped with the right-handed orientation. 
Consider a homeomorphism $\imath$ from $M\setminus P$  to the interior of $B^3$ that takes the biggest horizontal open disk in Figure~\ref{fig2-2} to the horizontal open disk bounded by the equator $E$ of $B^3$. Now we take the geometric completions of these open balls. 
Then the inverse map $\imath^{-1}$ induces a continuous map $\jmath$ from $B^3$ to $M$
that maps the boundary $S^2$ of $B^3$ to the flow-spine $P$.
The preimage of the singular set $S(P)$ of $P$ by the map $\jmath$ constitutes a trivalent graph $G$ on $S^2$ containing the equator $E$. The map $\jmath$ restricted to $S^2$, denoted by $\jmath'$, is called the {\it identification map}.

Let $R_1,R_2,\ldots,R_n$ be the regions of $P$ that are oriented so that the flow $\mathcal F$ is positively transverse to these regions. Let $S^+$ and $S^-$ denote the upper and lower hemispheres of $S^2$, respectively. The identification map $\jmath'$ satisfies the following properties:
\begin{itemize}
\item The preimage of each region $R_i$ by $\jmath'$ consists of two connected regions $R_i^+$ and $R_i^-$ bounded by $G$ each of which is homeomorphic to $R_i$ and which are contained in $S^+$ and $S^-$, respectively.
\item The orientation of $R_i^+$ (respectively, $R_i^-$) induced from that of $B^3$ coincides with (respectively, is opposite to) the orientation of $R_i$ through the map $\jmath'$.
\end{itemize}
The $3$-manifold $M$ is restored from $B^3$ by identifying the pairs of regions $R_i^+$ and $R_i^-$ for $i=1,\ldots,n$.

\begin{definition}
The trivalent graph $G$ on $S^2$ equipped with the identification map $\jmath'$ obtained from a flow-spine $P$ of $M$ as above is called the {\it DS-diagram} of $P$. The equator $E$ in $G$ is called the {\it $E$-cycle} of $G$. The $E$-cycle is oriented as the boundary of $S^+$ with the orientation induced from that of $B^3$.
\end{definition}

\begin{remark}
\begin{itemize}
\item[(1)] To be precise, the DS-diagram defined above is a {\rm DS-diagram with an $E$-cycle}.
If a simple spine is given then we may obtain its DS-diagram without $E$-cycles in the same manner.
There is a formal definition of $E$-cycles for DS-diagrams, see~\cite{Ish92}.
Note that a DS-diagram of a simple spine may have several possible positions of $E$-cycles.
In this paper, by a DS-diagram we mean a DS-diagram with a fixed $E$-cycle.
\item[(2)] Conversely, if a DS-diagram $G$ with an $E$-cycle is given, we may obtain a closed $3$-manifold $M$ from the diagram by using the identification map $\jmath'$. The image $\jmath'(S^2)$ in $\jmath(B^3)=M$ is a flow-spine of $M$.
\end{itemize}
\end{remark}

For convenience, taking the stereographic projection $\pi$ of $S^2$ from the south pole,
we describe the DS-diagram $G$ on the real plane $\Real^2$ so that $\pi(E)$ is the unit circle, $\pi(S^+)$ is the inside of $\pi(E)$ and $\pi(S^-)$ is the outside. The real plane $\Real^2$ is oriented such that it coincides with that of $S^2$ as the boundary of $B^3$.
The image $\pi(R_i^+)$ lies on $\pi(S^+)$ (respectively, $\pi(R_i^-)$ lies on $\pi(S^-)$)
and its orientation induced from that of $R_i$ coincides with (respectively, is opposite to) the orientation of $\Real^2$.
For simplicity, we denote $\pi(E)$, $\pi(R_i^+)$ and $\pi(R_i^-)$
by  $E$, $R_i^+$ and $R_i^-$, respectively. 

\begin{example}\label{ex0}
Consider the diagram described on the left in Figure~\ref{fig3-1}.
The $3$-manifold $M$ is obtained from $B^3$ by identifying $R_i^+$ and $R_i^-$ for $i=1,2$ so that the labeled edges along their boundary coincide.
Let $P$ be the simple spine obtained as the image of $S^2$ by the identification map $\jmath'$.
The edges with label $e_i$ are the preimage of the edge $e_i$ of $P$ by $\jmath'$
and the vertices with label $v$ are the preimage of the vertex $v$ of $P$. 
The right-top is the union of $\Nbd(S(P);P)$ and the region $R_2$, which is obtained from the union of $R_2^+$, $R_2^-$ and neighborhoods of the edges and the vertices by identifying them by $\jmath'$. The polyhedron $P$ is obtained from this branched polyhedron by attaching the region $R_1$ along the boundary.
Remark that the branched polyhedron on the right-top in Figure~\ref{fig3-1} is abstract, not an object embedded in $\Real^3$.
The polyhedron $P$ embedded in $\Real^3\subset S^3$ is described on the right-bottom. 
This polyhedron can be obtained from the $3$-ball by collapsing from the boundary. Hence it is a spine of $S^3$. 
Furthermore, setting a flow in $\Nbd(P;S^3)$ positively transverse to $P$ and extending it to $S^3\setminus P$ canonically as in Figure~\ref{fig2-2}, we may see that $P$ is a flow-spine of $S^3$.
This flow-spine is called the {\it positive abalone}.
The flow-spine obtained from the mirror of the DS-diagram on the left in Figure~\ref{fig3-1} is called the {\it negative abalone}. The abalone in Figure~\ref{fig3-1} is named ``positive'' 
since it supports the standard contact structure on $S^3$, which will be proved in a forthcoming paper.
The underlying abstract simple polyhedra of the positive and negative abalones are the same. This polyhedron first appeared in~\cite{Ike71}. See also~\cite{Ike72}.
\end{example}

\begin{figure}[htbp]
\begin{center}
\includegraphics[width=12cm, bb=130 417 580 713]{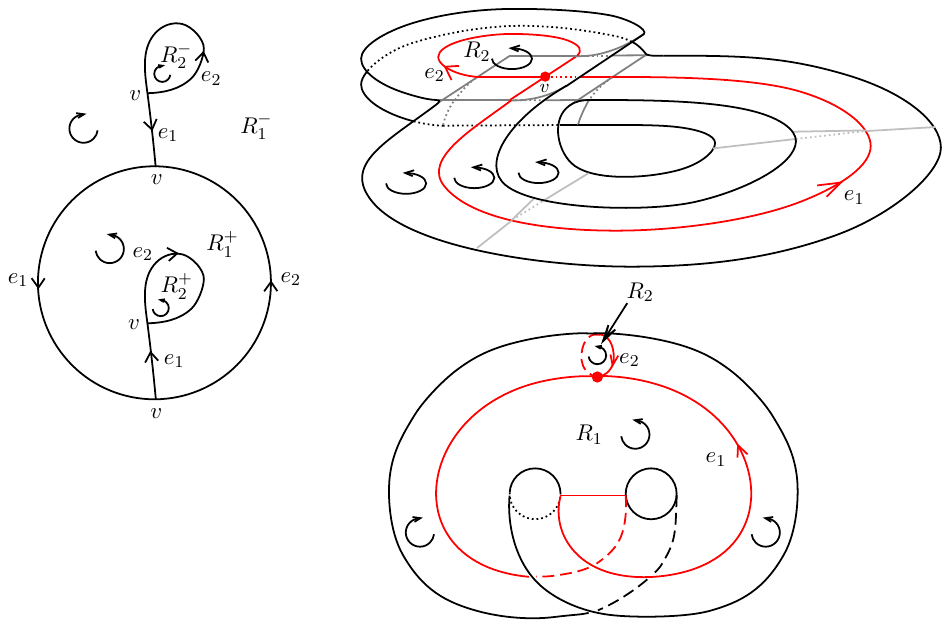}
\caption{The positive abalone. The left diagram is its DS-diagram, the right-top is the union of $\Nbd(S(P);P)$ and the region $R_2$ with branching structure, and the right-bottom is the positive abalone embedded in $S^3$.}\label{fig3-1}
\end{center}
\end{figure}

The following theorem is basic and important.

\begin{theorem}[\cite{Ish86}]\label{Ishii_FS}
Any pair $(M,\mathcal F)$ of a closed, connected, oriented, smooth $3$-manifold $M$ and a non-singular flow $\mathcal F$ on $M$ admits a flow-spine.
\end{theorem}

\subsection{Regular moves}

In this subsection, we introduce two kinds of moves of branched simple polyhedra.

\begin{definition}
\begin{itemize}
\item[(1)] The moves shown in Figure~\ref{fig3-7} and their mirrors are called  {\it first regular moves}~\cite{Ish92}
or {\it Matveev-Piergallini moves}~\cite{BP97}.
\item[(2)] The moves shown in Figure~\ref{fig3-8} and their mirrors are called {\it second regular moves}~\cite{Ish92} or {\it lune moves}~\cite{Mat03}.
\end{itemize}
\end{definition}

\begin{figure}[htbp]
\begin{center}
\includegraphics[width=9cm, bb=129 526 536 713]{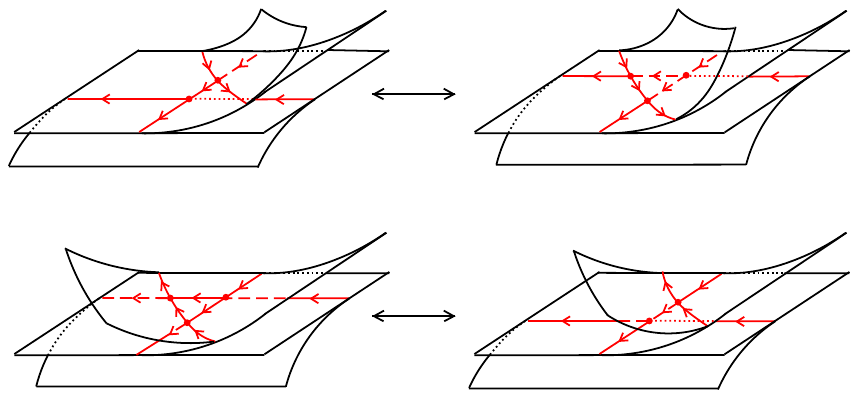}
\caption{First regular moves.}\label{fig3-7}
\end{center}
\end{figure}

\begin{figure}[htbp]
\begin{center}
\includegraphics[width=11cm, bb=129 605 475 712]{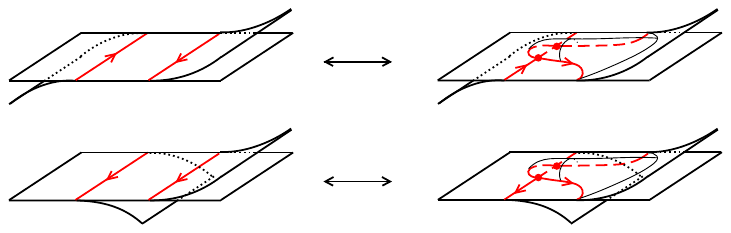}
\caption{Second regular moves.}\label{fig3-8}
\end{center}
\end{figure}

These moves are enough to trace deformation of non-singular flows in $3$-manifolds.
Two non-singular flows $\mathcal F_1$ and $\mathcal F_2$ in a closed, orientable $3$-manifold are said to be {\it homotopic} 
if there exists a smooth deformation from $\mathcal F_1$ to $\mathcal F_2$ in the set of non-singular flows.

\begin{theorem}[\cite{Ish92}]\label{thm_ishii}
Let $P_1$ and $P_2$ be flow-spines of a closed, orientable $3$-manifold $M$.
Let $\mathcal F_1$ and $\mathcal F_2$ be flows on $M$ carried by $P_1$ and $P_2$, respectively. 
Suppose that $\mathcal F_1$ and $\mathcal F_2$ are homotopic.
Then $P_1$ is obtained from $P_2$ by applying first and second regular moves successively.
\end{theorem}

\begin{remark}
We may apply the moves from the left to the right in Figures~\ref{fig3-7} and~\ref{fig3-8} with fixing the flow carried by the flow-spine. On the other hand, we may need to homotope the flow when we apply the moves from the right to the left. 
Sometimes, such homotopy move cannot be obtained in the set of Reeb flows, see the proof of Theorem~\ref{thm00} below.
\end{remark}

\section{Admissibility and positivity}

\subsection{Admissibility condition}

Let $M$ be a closed, connected, oriented $3$-manifold.
Let $P$ be a branched simple spine of $M$ and $R_1,\ldots,R_n$ be the regions of $P$.
Assign orientations to these regions so that they satisfy the branching condition.
Let $\bar R_i$ be the metric completion of $R_i$ with the path metric inherited from a Riemannian metric on $R_i$. Let $\kappa_i:\bar R_i\to M$ be the natural extension of the inclusion $R_i\to M$. Assign an orientation to each edge of $P$ in an arbitrary way.

\begin{definition}
A branched simple spine $P$ is said to be {\it admissible} if there exists an assignment of real numbers $x_1,\ldots,x_m$ to the edges $e_1,\ldots,e_m$, respectively, of $P$ such that
for any $i\in\{1,\ldots,n\}$
\begin{equation}\label{eq1000}
   \sum_{\tilde e_j\subset \partial \bar R_i}\varepsilon_{ij} x_j>0,
\end{equation}
where $\tilde e_j$ is an open interval or a circle on $\partial \bar R_i$ such that
$\kappa_i|_{\tilde e_j}:\tilde e_j\to e_j$ is a homeomorphism, and $\varepsilon_{ij}=1$ if the orientation of $e_j$ coincides with that of $\kappa_i(\tilde e_j)$ induced from 
the orientation of $R_i$ and $\varepsilon_{ij}=-1$ otherwise.
\end{definition}

\begin{example}\label{ex1}
Let $P$ be the positive abalone in Figure~\ref{fig3-1} embedded in $S^3$.
Suppose that there exists a contact structure supported by $P$, that is, there exists a contact form $\alpha$ on $S^3$ whose Reeb flow is carried by $P$.
The abalone $P$ is obtained from the branched polyhedron described on the right-top in the figure by attaching the disk corresponding to the region $R_1$.
Since the Reeb flow is transverse to the regions $R_1$ and $R_2$, the integrated values
$\int_{\partial R_1}\alpha$ and $\int_{\partial R_2}\alpha$ should be positive, where the orientation of $\partial R_i$ is induced from that of $R_i$ for each $i=1,2$.
From the figure, we see that $\partial R_1=e_1+2e_2$ and $\partial R_2=-e_2$.
Setting $x_1=\int_{e_1}\alpha$ and $x_2=\int_{e_2}\alpha$, we have the inequalities
$x_1+2x_2>0$ and $x_2<0$. Therefore, the existence of an assignment of real numbers $x_1, x_2$
satisfying these inequalities to the edges $e_1, e_2$ is a necessary condition for $\Nbd(P;S^3)$ to have a contact form whose Reeb flow is positively transverse to $P$. This is the admissibility condition. Since $(x_1,x_2)=(3,-1)$ satisfies the two inequalities, the positive abalone $P$ is admissible.
\end{example}

There are many flow-spines that satisfy the admissibility condition. 
Below, we show that any branched special spine of a rational homology $3$-sphere satisfies the condition.

Let $m'$ be the rank of $H_1(S(P);\Integer)$.
Take a maximal tree $T$ of $S(P)$ and assign $0$ to the edges on $T$.
Let $e_{j_1},\ldots,e_{j_{m'}}$ be the edges on $S(P)$ not contained in $T$ and set 
\[
   a_{ij}=\sum_{\tilde e_j\subset \partial \bar R_i}\varepsilon_{ij}
\]
for $i=1,\ldots,n$ and $j=j_1,\ldots,j_{m'}$. In this setting, $P$ is admissible if and only if the set
\[
   C(P)=\{(x_1,\ldots,x_{m'})\in\Real^{m'}\mid a_{ij_1}x_1+\cdots+a_{ij_{m'}}x_{m'}>0,\;\,i=1,\ldots,n\}
\]
is non-empty.

Let $A$ be the $n\times m'$ matrix with the entries $a_{ij}$.
If the spine is special then $A$ is a 
square 
matrix.

\begin{lemma}\label{lemma20}
Let $P$ be a branched special spine of $M$.
Then $\det(A)\ne 0$ if and only if $M$ is a rational homology $3$-sphere.
\end{lemma}

\begin{proof}
The manifold $M$ has a cell decomposition 
whose $0$-cells are the vertices of $S(P)$, $1$-cells correspond to the edges of $S(P)$,
$2$-cells correspond to the open disks $P\setminus S(P)$ and 
$3$-cell corresponds to the open $3$-ball $M\setminus P$.
Consider the chain complex of this decomposition:
\[
    0\to C_3
    \overset{\partial_3}{\longrightarrow} C_2
    \overset{\partial_2}{\longrightarrow} C_1
    \overset{\partial_1}{\longrightarrow} C_0 \to 0,
\]
where $C_i$ is the chain of $i$-dimensional cells.
The map $\partial_3$ is the zero map because each $2$-cell appears twice on the boundary of
the $3$-cell with opposite orientations. The map $\partial_2$ is given by the matrix $A$.
Since $H_2(M;\Integer)$ has no torsion, $\det(A)\ne 0$ if and only if $H_2(M;\Integer)=0$.
By Poincar\'{e} duality and the universal coefficient theorem, $H_2(M;\Integer)=0$ if and only if
$H_i(M;\Rational)=H_i(S^3;\Rational)$ for $i\in\Integer$.
Thus the assertion follows.
\end{proof}

\begin{proposition}\label{lemma21}
A branched special spine of a rational homology $3$-sphere is admissible.
\end{proposition}

\begin{proof}
The set $C(P)$ is the union of the $m'$ open half spaces in $\Real^{m'}$ given by
$a_{ij_1}x_1+\cdots+a_{ij_{m'}}x_{m'}>0$ for $i=1,\ldots,m'$.
Since we are considering a rational homology $3$-sphere, we have $\det(A)\ne 0$ by 
Lemma~\ref{lemma20}. Therefore the vectors ${}^t(a_{ij_1},\ldots,a_{ij_{m'}})$, $i=1,\ldots,m'$, are
linearly independent and hence $C(P)$ is non-empty.
\end{proof}

\begin{example}\label{ex2}
Let $S$ be an oriented $2$-sphere with two disjoint oriented disks $D_1$ and $D_2$ in $S$
and let $P$ be the branched polyhedron obtained from $S$ by identifying $D_1$ and $D_2$ by an orientation-preserving diffeomorphism. Then $P$ is a flow-spine of $S^2\times S^1$ 
without vertices (cf.~\cite[Remark~1.2]{EI05}).
The flow described in Figure~\ref{fig3-5} is a flow carried by $P$. 
This flow-spine satisfies $C(P)=\emptyset$ and hence it is not admissible.
In particular, there does not exist a contact form on $S^2\times S^1$ whose Reeb flow is positively transverse to $P$.
\end{example}

\begin{figure}[htbp]
\begin{center}
\includegraphics[width=3.7cm, bb=234 545 351 713]{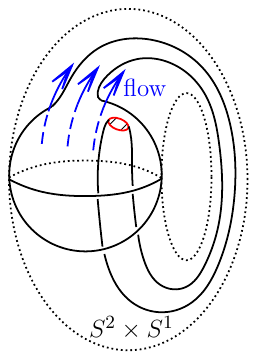}
\caption{A flow-spine of $S^2\times S^1$ with $C(P)=\emptyset$.}\label{fig3-5}
\end{center}
\end{figure}

\subsection{Necessity of positivity}

As mentioned in the introduction, we define the positivity of a flow-spine as follows.

\begin{definition}
A flow-spine is said to be {\it positive} if it has at least one vertex and all vertices are of $\ell$-type. 
\end{definition}

Although there is a flow-spine that is not positive but carries a Reeb flow, we need to restrict our discussion to positive flow-spines. The reason is that either the existence or the uniqueness of a contact manifold for a given flow-spine does not hold at least in some case as we will prove in the following theorem.
Recall that a contact structure $\xi$ on $M$ is said to be {\it supported} by a flow-spine $P$ if there exists a contact form $\alpha$ such that $\xi=\ker\alpha$ and its Reeb flow is carried by $P$.

\begin{theorem}\label{thm00}
Suppose that $M$ admits a tight contact structure. Then one of the following holds:
\begin{itemize}
\item[(1)] There exists a flow-spine of $M$ that does not support any contact structure.
\item[(2)] There exists a flow-spine of $M$ supporting two contact structures that are not contactomorphic.
\end{itemize}
\end{theorem}

\begin{proof}
Let $\xi_0$ and $\xi_1$ be contact structures on $M$ that belong to the same homotopy class of $2$-plane fields. Let $\mathcal F_0$ and $\mathcal F_1$ be Reeb flows generated by Reeb vector fields of $\xi_0$ and $\xi_1$, respectively. These flows are homotopic.
Let $P_0$ and $P_1$ be flow-spines that carry $\mathcal F_0$ and $\mathcal F_1$, respectively.

Suppose that $\xi_0$ is tight, and let $\xi_1$ be an overtwisted contact structure on $M$ obtained from $\xi_0$ by applying a Lutz full-twist~\cite{Lut77}. Note that a Lutz full-twist does not change the homotopy class of $\xi_0$.
By Theorem~\ref{thm_ishii}, there exists a sequence $P_0=Q_1, Q_2, \ldots, Q_m=P_1$ of flow-spines such that $Q_{i+1}$ is obtained from $Q_i$ by some regular move for $i=1,\ldots,m-1$.
Now assume that both of the statements~(1) and~(2) in the assertion do not hold. 
Then each $Q_i$ supports a unique contact structure up to contactomorphism.
Since the regular moves from the left to the right in Figures~\ref{fig3-7} and~\ref{fig3-8} can be applied with fixing the flows, the contact structures supported by $Q_i$ and $Q_{i+1}$ are contactomorphic. Hence $\xi_0$ and $\xi_1$ are contactomorphic and this is a contradiction.
\end{proof}

\begin{remark}
The positive abalone in Figure~\ref{fig3-1} is a positive flow-spine.
The flow-spine in Example~\ref{ex2} carries a flow transverse to $S^2\times \{t\}$ for any $t\in S^1$.
Such a flow cannot be a Reeb flow. This flow-spine is not positive by definition.
\end{remark}

\subsection{Admissibility of positive flow-spines}

In this subsection, we prove the following proposition. 

\begin{proposition}\label{prop01}
Any positive flow-spine satisfies the admissibility condition.
\end{proposition}

Let $P$ be a positive flow-spine.
In the DS-diagram of $P$, the preimage of each edge $e_j$ of $P$ by the identification map $\jmath'$ consists of three copies of $e_j$, one lies on the $E$-cycle, another one lies in the inside of the $E$-cycle and the last one lies outside.
The diagram in a neighborhood of an edge on the $E$-cycle is one of the four cases described on the top in Figure~\ref{fig3-3},
the corresponding part of the spine is described in the middle, 
and the diagram in a neighborhood of the corresponding edge in the inside of the $E$-cycle is as shown on the bottom. The $E$-cycles in the top and bottom figures are described by dashed lines.

\begin{figure}[htbp]
\begin{center}
\includegraphics[width=14cm, bb=129 506 537 712]{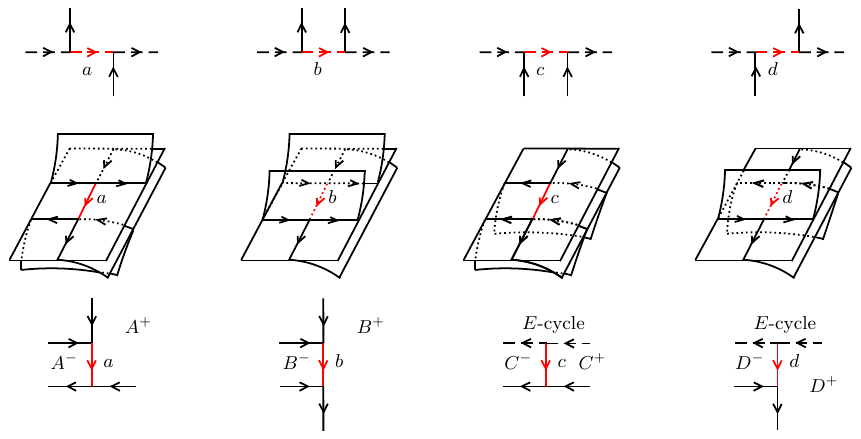}
\caption{Neighborhoods of edges on a DS-diagram.}\label{fig3-3}
\end{center}
\end{figure}

We put the labels $A^-$ and $A^+$ to the two sides of each edge of type $a$ as shown on the bottom in Figure~\ref{fig3-3} and we put the labels $B^-, B^+, C^-, C^+, D^-, D^+$ in the same manner.

\begin{lemma}\label{lemma3-1}
A region in the inside of the $E$-cycle in a DS-diagram and not adjacent to the $E$-cycle 
has only label $A^-$.
\end{lemma}

\begin{proof}
Let $R^+$ be a region in the assertion.
From the bottom figures in Figure~\ref{fig3-3}, we see that $\partial R^+$ does not contain edges of type $c$ and $d$. Hence $R^+$ has only labels $A^-, A^+, B^-$ and $B^+$. 
We observe how the orientations on the edges change when we travel along $\partial R^+$.
We say that the orientation of an edge is ``consistent'' if it coincides with that of $\partial R^+$
and ``opposite'' otherwise.
From the bottom figures in Figure~\ref{fig3-3}, we see that the orientation changes from ``consistent'' to ``opposite'' after passing through the sides $A^+$ and $B^-$ and there is no case where the orientation changes from ``opposite'' to ``consistent''. Hence $R^+$ cannot have labels $A^+$ and $B^-$. Moreover, since the orientations of the edges near the sides $A^-$ and $B^+$ are ``opposite'' and ``consistent'', respectively, $R^+$ can have either only label $A^-$ or only label $B^+$.
Suppose that it has only label $B^+$. Then, from the middle figure in Figure~\ref{fig3-3}, we see that $\jmath'(\partial R^+)$ corresponds to the $E$-cycle. However, in the same figure, we see that there is an edge of $P$ that is not contained in $\jmath'(\partial R^+)$.
Since $S(P)$ is the image of the $E$-cycle by $\jmath'$, we have a contradiction.
\end{proof}

Now we prove the proposition.

\begin{proof}[Proof of Proposition~\ref{prop01}]
We set the assignment $x$ of a real number $x(e_j)$ to each edge $e_j$ of the singular set $S(P)$ of a positive flow-spine $P$ as 
\[
\begin{cases}
x(e_j)=-\frac{1}{2} & \text{if $e_j$ is of type $a$} \\
x(e_j)=2 & \text{if $e_j$ is of type $b$} \\
x(e_j)=3 & \text{if $e_j$ is of type $c$} \\
x(e_j)=4 & \text{if $e_j$ is of type $d$.}
\end{cases}
\]
By Lemma~\ref{lemma3-1}, 
if a region $R^+$ lies inside the $E$-cycle and is not adjacent to the $E$-cycle
then $\sum_{\tilde e_j\subset \partial\bar R}\varepsilon_{ij}x(e_j)>0$.

Let $R^+$ be a region of the DS-diagram of $P$ lying inside the $E$-cycle and adjacent to the $E$-cycle.
The boundary $\partial R^+$ of $R^+$ consists of an alternative sequence
$p_1, q_1, p_2, q_2, \ldots, p_m, q_m$ of paths $p_1,\dots, p_m$ on the $E$-cycle and paths $q_1,\ldots,q_m$ lying inside the $E$-cycle, see Figure~\ref{fig3-9}.
The orientation of each edge on the path $p_i$ is consistent with the orientation of $\partial R^+$. The starting edge of the path $q_i$ is of type $c$ if the label on the region $R^+$ is $C^+$ and of type $d$ if it is $D^+$. Their orientations are consistent with that of $\partial R^+$.
We follow $q_i$ from the starting edge and check if the orientations of the edges are consistent with or opposite to that of $\partial R^+$. From the four figures on the bottom in Figure~\ref{fig3-3}, we see that the orientation changes from ``consistent'' to ``opposite'' 
when $\partial R^+$ passes through an edge with label $A^+$, $B^-$, $C^+$ or $D^-$, and there is no edge where the orientation changes from ``opposite'' to ``consistent''. This means that the path $q_i$ divides into two paths $q'_i$ and $q''_i$ so that the orientations of all edges on $q_i'$ are consistent with that of $\partial R^+$ and those of all edges of $q_i''$ are opposite.

\begin{figure}[htbp]
\begin{center}
\includegraphics[width=6cm, bb=174 507 401 712]{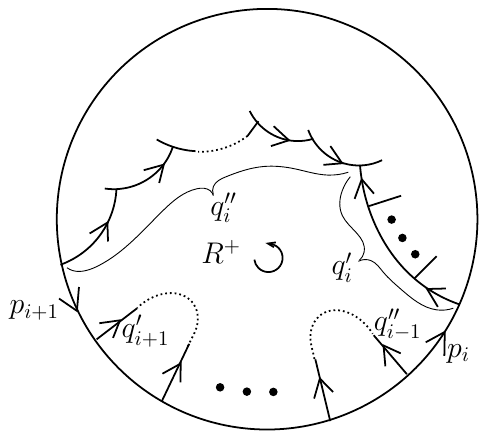}
\caption{The paths $p_i$ and $q_i$, where $q_i$ is the concatenation of $q_i'$ and $q_i''$.}\label{fig3-9}
\end{center}
\end{figure}

We are going to observe which types of edges can appear on each of paths $p_i$ and $q_i$. 
By the four figures on the top in Figure~\ref{fig3-3},  the sequence of types of the edges of $p_i$ is either $b$ or $ac^{n_1}d$, where $n_1\geq 0$. The sum $x(p_i)$ of the assignment $x$ to the edges on $p_i$ is
\[
x(p_i)=
\begin{cases}
2 & \text{if the type is $b$} \\
-\frac{1}{2}+3n_1+4=3n_1+\frac{7}{2} & \text{if the type is $ac^{n_1}d$.}
\end{cases}
\]
By the four figures on the bottom in Figure~\ref{fig3-3},  
the sequence of types of the edges of $q_i'$ is either $c$ or $db^{n_2}a$, where $n_2\geq 0$, and the sum $x(q_i')$ of the assignment $x$ to the edges on $q_i'$ is 
\[
x(q_i')=
\begin{cases}
3 & \text{if the type is $c$} \\
4+2n_2-\frac{1}{2}=2n_2+\frac{7}{2} & \text{if the type is $db^{n_2}a$.} 
\end{cases}
\]
By the same figures,
the sequence of types of the edges of $q_i''$ is either $d^{-1}$ or $b^{-1}a^{-n_3}c^{-1}$, where $n_3\geq 0$, and the sum $x(q_i'')$ of the assignment $x$ to the edges on $q_i''$ is 
\[
x(q_i'')=
\begin{cases}
-4 & \text{if the type is $d^{-1}$} \\
-2+\frac{1}{2}n_3-3=\frac{1}{2}n_3-5  & \text{if the type is $b^{-1}a^{-n_3}c^{-1}$.} 
\end{cases}
\]
The minimum values of $x(p_i)$, $x(q_i')$ and $x(q_i'')$ are $2$, $3$ and $-5$, respectively.
Hence the sum of real numbers on the edges on $\partial R^+$ is positive unless 
the types of $p_i$, $q_i'$ and $q_i''$ are $b$, $c$ and $b^{-1}c^{-1}$, respectively, for all $i=1,\ldots,m$.

Suppose that the boundary $\partial R^+$ of $R^+$ consists of a sequence of edges 
\begin{equation}\label{eqxxx}
   b_1c_1b_2^{-1}c_2^{-1} b_3c_3b_4^{-1}c_4^{-1}\cdots b_{2\ell-1}c_{2\ell-1}b_{2\ell}^{-1}c_{2\ell}^{-1},
\end{equation}
where $b_j$ is an edge of type $b$ and $c_j$ is an edge of type $c$. 
For each $j=1,\ldots,\ell$, we set a relation $\succ$ between two edges $b_{2j-1}$ and $b_{2j}$ as $b_{2j-1}\succ b_{2j}$. 
For each region inside the $E$-cycle whose boundary is given in the above form,
we give the relation $\succ$ to the edges of type $b$ in the same manner.

Now assume that there exists a sequence of edges $e_1, e_2, \ldots, e_k$ of $P$ of type $b$ that gives a cyclic order as $e_1\succ e_2\succ\cdots\succ e_k\succ e_{k+1}=e_1$.
For each $s=1,\ldots,k$, the relation $e_s\succ e_{s+1}$ implies that there exists a region whose boundary contains the sequence $e_sf_se_{s+1}^{-1}f_{s+1}^{-1}$, where $f_s$ and $f_{s+1}$ are edges of type $c$.
However, this is impossible since the edges $f_1,\ldots, f_{k-1}$ appear on the $E$-cycle as shown in Figure~\ref{fig3-4} and we cannot set $f_k$ for $e_k\succ e_1$.
This means that $\succ$ is a partial order for the edges of $P$ of type $b$.

\begin{figure}[htbp]
\begin{center}
\includegraphics[width=6cm, bb=159 472 412 713]{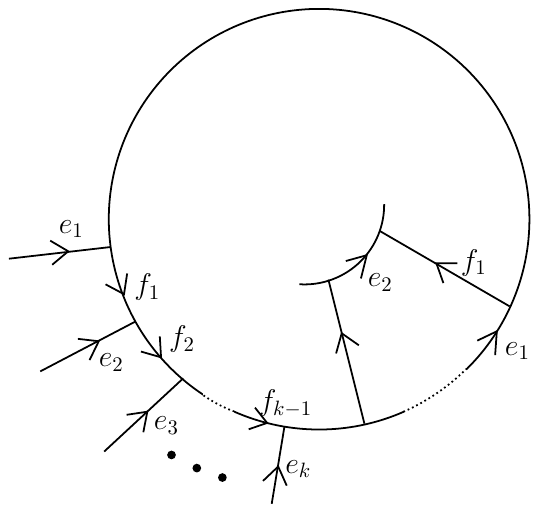}
\caption{Positions of $e_s$ for $s=1,\ldots,k$ and $f_s$ for $s=1,\ldots,k-1$.}\label{fig3-4}
\end{center}
\end{figure}

Let $\{e_1,\ldots, e_{n_b}\}$ be the set of edges of $P$ of type $b$, equipped with the partial order $\succ$.
We modify the assignment $x$ to 
these edges as $x(e_j)=2+\delta_j$,
where $\delta_j$'s are sufficiently small positive real numbers such that 
$\delta_{j_1}>\delta_{j_2}$ if $e_{j_1}\succ e_{j_2}$.
Then the sum of the real numbers assigned to the edges in~\eqref{eqxxx} becomes positive.
This shows that the modified assignment satisfies the admissibility condition.
\end{proof}

\section{Reference $1$-forms of positive flow-spines}

Let $M$ be a closed, connected, oriented, smooth $3$-manifold and $P$ be a positive flow-spine of $M$.
In this section, we define a $1$-form $\eta$ on $M$ associated with $P$ such that $\eta\land d\eta\geq 0$ using the decomposition~\eqref{decomp}.

\subsection{Decomposition and gluing maps}

For a positive flow-spine $P$ of $M$, set
$Q=\Nbd (S(P);P)$ and 
let $\mathcal R_i$ be the connected component of $P\setminus \Int Q$ contained in the region $R_i$ of $P$.
Let $\Nbd(S(P);M)$ be a neighborhood of $S(P)$ in $M$,
choose a neighborhood $N_P=\Nbd(P;M)$ of $P$ in $M$ sufficiently thin with respect to $\Nbd(S(P);M)$ and set $N_Q=\Nbd(P;M)\cap \Nbd(S(P);M)$. 
Then $M$ decomposes as
\begin{equation}\label{decomp}
   M=N_P\cup N_D=\left(N_Q\cup\bigcup_{i=1}^n N_{\mathcal R_i}\right)\cup N_D,
\end{equation}
where $N_{\mathcal R_i}$ is the connected component of the closure of $N_P\setminus N_Q$ containing $\mathcal R_i$, 
which is diffeomorphic to $\mathcal R_i\times [0,1]$, and $N_D$ is the closure of $M\setminus N_P$, which is a $3$-ball.
By identifying the $3$-ball $B^3$ in Section~\ref{sec22} with $N_D$, we draw 
the $E$-cycle on $\partial N_D$ and regard $N_D$ as $D^2\times [0,1]$, where $D^2$ is the unit $2$-disk, so that a tubular neighborhood of the $E$-cycle on $\partial N_D$ corresponds to $\partial D^2\times [0,1]$.

To define the gluing maps of these pieces, we re-choose $N_{\mathcal R_i}$ and $N_D$ slightly larger.
For $i=1,\ldots,n$, we denote the gluing map of $N_{\mathcal R_i}$ to $N_Q$ by $g_{\mathcal R_i}$,
which is a diffeomorphism $g_{\mathcal R_i}:N_{\mathcal R_i'}\to N_Q$ to the image, 
where $\mathcal R_i'=\Nbd(\partial \mathcal R_i;\mathcal R_i)$ and $N_{\mathcal R_i'}=\mathcal R_i'\times [0,1]\subset N_{\mathcal R_i}$.
Also, we denote the gluing map of $N_D$ to $N_P$ by $g_D$, which is a diffeomorphism 
$g_D:\Nbd(\partial N_D;N_D)\to N_P$ to the image.

We may assume that a flow carried by $P$ is positively transverse to $D^2\times \{t_0\}$ on $N_D$ and $\mathcal R_i\times \{t_0\}$ on $N_{\mathcal R_i}$ for any $t_0\in[0,1]$.

\subsection{Construction on $N_Q$}\label{sec52}

First we define a $1$-form on $N_Q$. Let $Q'$ be an oriented compact surface with an embedded graph $G'$ such that 
\begin{itemize}
\item[(1)] $Q'$ collapses onto $G'$ and 
\item[(2)] there exists a continuous map $\pr_Q:Q\to Q'$ such that
$\pr_Q|_{S(P)}$ is an embedding of the graph $S(P)$ into $Q'$ with $G'=\pr_Q(S(P))$ and
$\pr_Q|_{Q\setminus S(P)}$ is an orientation-preserving local diffeomorphism.
\end{itemize}
We then fix a projection $\pr:N_Q\to Q'$ such that $\pr|_Q=\pr_Q$.

Next we decompose $Q'$ into rectangles $U^e_1,\ldots,U^e_{n_e}$ corresponding to the images of neighborhoods of the edges $e_1,\ldots,e_{n_e}$ of $P$ and rectangles $U^v_1,\ldots,U^v_{n_v}$ 
corresponding to those of vertices $v_1,\ldots,v_{n_v}$ of $P$ as shown in Figure~\ref{fig4}.

\begin{figure}[htbp]
\begin{center}
\includegraphics[width=10.0cm, bb=130 568 546 713]{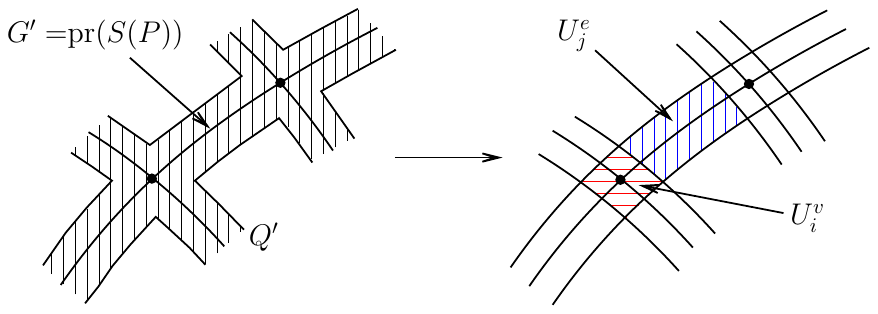}
\caption{Rectangular decomposition of $Q'$.}\label{fig4}
\end{center}
\end{figure}

For each $U^e_i$, the preimage $B^e_i=\pr^{-1}(U^e_i)$ has a shape as shown on the left in Figure~\ref{fig5}. According to the order determined by the flow, we denote the horizontal faces of $B^e_i$ by $H^e_{i,1}$, $H^e_{i,2}$, $H^e_{i,3}$ and $H^e_{i,4}$. 
These faces are indicated on the right in the figure, where 
$H_{i,2}^e$ and $H_{i,4}^e$ are top faces of $B_i^e$ and $H_{i,1}^e$ and $H_{i,3}^e$ are bottom faces of $B_i^e$. Set $\ell_i=\pr(S(P)\cap B_i^e)$.
The complement $U^e_i\setminus\Int\Nbd(\ell_i;U^e_i)$ of an open neighborhood of $\ell_i$ in 
$U_i^e$ consists of two connected components. 
We denote by $U^e_{i,1}$ the component whose preimage by $\pr$ is connected, and by 
$U^e_{i,2}$ the component whose preimage has two connected components.
Set $B^e_{i,j}=\pr^{-1}(U^e_{i,j})$ for $j=1,2$.
We fix coordinates $(x_i^e,y_i^e,t^e_i)$ on $B^e_i$ such that
\begin{itemize}
\item 
$B_{i,2}^e\subset \{0\leq y_i^e\leq 1\}$, $B_{i,1}^e\subset \{2\leq y_i^e\leq 3\}$,
\item 
$H^e_{i,2}\subset \{t_i^e=1\}$, $H^e_{i,3}\subset \{t_i^e=2\}$, $H^e_{i,4}\subset \{t_i^e=3\}$,
\item 
$H^e_{i,1}\subset \{t_i^e=h(y_i^e)\}$, where $h:[0,3]\to [0,2]$ is a monotone increasing smooth function such that $h(y_i^e)=0$ for $0\leq y_i^e\leq 1$ and $h(y_i^e)=2$ for $2\leq y_i^e\leq 3$.
\end{itemize}
We can think that $B^e_i$ is embedded in the cube $[0,L]\times [0,3]\times [0,3]$ with 
the coordinates $(x^e_i, y^e_i, t^e_i)$, where $L$ is a positive real number.
Note that $L$ should be large enough so that $N_P$ can be constructed from the pieces $N_Q$ and $N_{\mathcal R_i}$ for $i=1,\ldots,n$.

\begin{figure}[htbp]
\begin{center}
\includegraphics[width=12.5cm, bb=130 528 557 710]{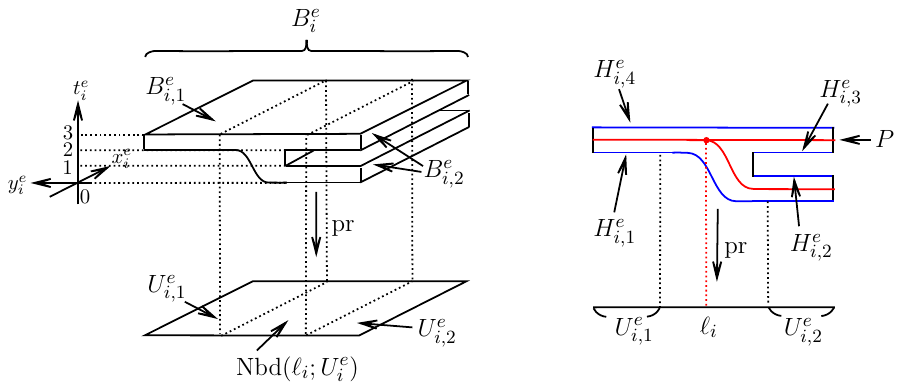}
\caption{The neighborhood $B^e_i$ of a triple line $e_i$ of $P$.}\label{fig5}
\end{center}
\end{figure}

Since $P$ is a positive flow-spine, each vertex is of $\ell$-type. 
For each $U^v_i$, the preimage $B^v_i=\pr^{-1}(U^v_i)$ has a shape as shown on the left in Figure~\ref{fig6}.
According to the order determined by the flow, we denote the horizontal faces of 
$B^v_i$ by $H^v_{i,1}, H^v_{i,2'}, H^v_{i,3'}, H^v_{i,2}, H^v_{i,3}, H^v_{i,4}$.
These faces are indicated on the left in the figure, where $H^v_{i,2}$,  $H^v_{i,2'}$ and  $H^v_{i,4}$ are top faces of  $B^v_i$ drawn in red, and  $H^v_{i,1}$,  $H^v_{i,3}$ and  $H^v_{i,3'}$ are bottom faces of $B_i^v$ drawn in blue.
The image of $S(P)\cap B_i^v$ by $\pr$ consists of two lines $\ell_i$ and $\ell_i'$ intersecting transversely at the image $\pr(v_i)$ of the vertex $v_i$ in $B_i^v$. We set $\ell_i$ to be the line that does not intersect $\pr(H_{i,2}^v)$ and $\ell_i'$ to be the other line,
which does not intersect $\pr(H_{i,2'}^v)$.
We denote by $U^v_{i,1}$ the connected component of the complement $U^v_i\setminus\Int\Nbd(\ell_i;U^v_i)$ of an open neighborhood of $\ell_i$ whose preimage by $\pr$ is connected
and by $U^v_{i,2}$ the component of the complement whose preimage by $\pr$ is not connected. 
Similarly, we denote by $U^v_{i,1'}$ the connected component of the complement $U^v_i\setminus\Int\Nbd(\ell'_i;U^v_i)$ of an open neighborhood of $\ell'_i$ whose preimage by $\pr$ is connected and by $U^v_{i,2'}$ the component of the complement whose preimage by $\pr$ is not connected.
We set $B^v_{i,j}=\pr^{-1}(U^v_{i,j})$ for $j=1,2,1',2'$.

Set coordinates $(x^v_i, y^v_i, t^v_i)$ on $B^v_i$ such that
\begin{itemize}
\item $B_{i,2}^v\subset \{0\leq y_i^v\leq 1\}$, $B_{i,1}^v\subset \{2\leq y_i^v\leq 3\}$,
\item $B_{i,2'}^v\subset \{0\leq x_i^v\leq 1\}$, $B_{i,1'}^v\subset \{2\leq x_i^v\leq 3\}$,
\item $H^v_{i,2}\subset \{t_i^v=3\}$, $H^v_{i,3}\subset \{t_i^v=4\}$, $H^v_{i,4}\subset \{t_i^v=5\}$,
\item $H^v_{i,3'}\subset \{t_i^v=h(y_i^v)+2\}$, $H^v_{i,2'}\subset \{t_i^v=h(y_i^v)+1\}$,
\item $H^v_{i,1}\subset \{t_i^v=\hat h(x_i^v,y_i^v)\}$, where $\hat h:[0,3]\times [0,3]\to[0,4]$ is a smooth function such that $\partial h/\partial x_i^v\geq 0$, $\partial h/\partial y_i^v\geq 0$, $\hat h(x_i^v,y_i^v)=h(y_i^v)$ for $(x_i^v,y_i^v)\in U^v_{i,2'}$,  $\hat h(x_i^v,y_i^v)=h(y_i^v)+2$ for $(x_i^v,y_i^v)\in U^v_{i,1'}$,  $\hat h(x_i^v,y_i^v)=h(x_i^v)$ for $(x_i^v,y_i^v)\in U^v_{i,2}$ and $\hat h(x_i^v,y_i^v)=h(x_i^v)+2$ for $(x_i^v,y_i^v)\in U^v_{i,1}$.
\end{itemize}
We can think that $B^v_i$ is embedded in the cube $[0,3]\times [0,3]\times [0,5]$ with 
the coordinates $(x^v_i, y^v_i, t^v_i)$.

\begin{figure}[htbp]
\begin{center}
\includegraphics[width=14cm, bb=130 534 558 711]{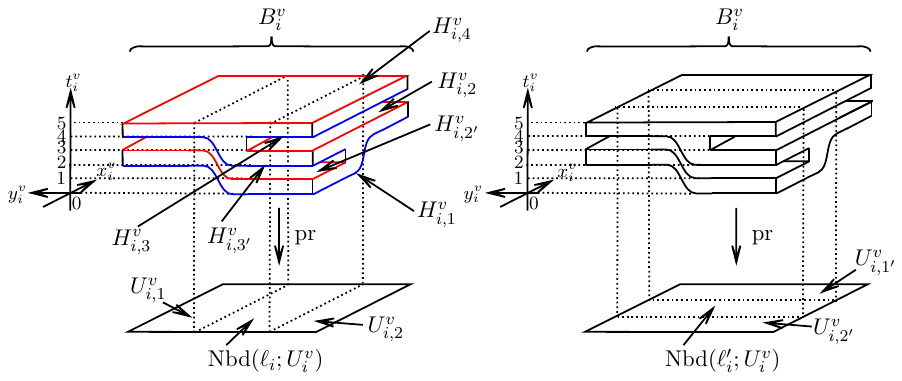}
\caption{The neighborhood $B^v_i$ of a vertex $v_i$ of $P$ of $\ell$-type.}\label{fig6}
\end{center}
\end{figure}

On each $B^k_i$, where $k=e$ or $v$, we define a $1$-form $\eta^k_i$ on $B^k_i$ by $\eta^k_i=dt^k_i$.
We can glue the $1$-forms $\eta^e_i$ for $i=1,\ldots,n_e$ and $\eta^v_j$ for $j=1,\ldots,n_v$ canonically. We denote the $1$-form on $N_Q$ obtained in this way by $\eta_Q$.

\subsection{Construction on $N_P$}

Recall that the $3$-manifold $M$ is obtained from $N_Q$ by attaching $N_{\mathcal R_i}$, $i=1,\ldots,n$, 
where $P\cap \mathcal R_i$ is regarded as $\mathcal R_i\times\{1/2\}\subset N_{\mathcal R_i}$, 
and then filling the boundary of $N_P$ by $N_D$.

Set $N_{\mathcal R_i}=\mathcal R_i\times [0,1]$ and let $t_i$ be the coordinate of $[0,1]$.
We choose the gluing map $g_{\mathcal R_i}$ of $N_{\mathcal R_i}$ to $N_Q$ so that
\begin{itemize}
\item[(e-gl)] the gluing part of $B^e_j$ with $\bigcup_{i=1}^n N_{\mathcal R_i}$ is exactly $B_{j,1}^e\cup B_{j,2}^e$,
\item[(e-01)] $t_j^e\circ g_{\mathcal R_i}=t_i$ on $g^{-1}_{\mathcal R_i}(B_{j,2}^e\cap \{0\leq t_j^e\leq 1\})$,
\item[(e-23)] $t_j^e\circ g_{\mathcal R_i}=t_i+2$ on $g^{-1}_{\mathcal R_i}((B_{j,1}^e\cup B_{j,2}^e)\cap \{2\leq t_j^e\leq 3\})$,
\item[(v-gl)] the gluing part of $B_j^v$ with $\bigcup_{i=1}^n N_{\mathcal R_i}$ is exactly the union of $B_{j,p}^v\cap B_{j,q}^v$, where $p$ and $q$ run over $\{ 1,2,1',2'\}$ with $p\ne q$, $B_{j,2}^v\cap\{4\leq t_j^v\leq 5\}$ and 
$W_{j}^v$ with rounding the corners along $B_j^v\cap \pr^{-1}(\partial \Nbd(\ell_j;U_j^v) \cap \partial \Nbd(\ell_j';U_j^v))$ smoothly, where
\[
W_{j}^v=B_{j,2'}^v\cap \pr^{-1}(\Nbd(\ell_j;U_j^v))\cap \{h(y_j^v)\leq t_j^v \leq h(y_j^v)+1\},
\]
see the right in Figure~\ref{fig7}.
\item[(v-01)] $t_j^v\circ g_{\mathcal R_i}=t_i$ on $g^{-1}_{\mathcal R_i}((B_{j,2}^v\cap B_{j,2'}^v)\cap \{0\leq t_j^v\leq 1\})$,
\item[(v-23)] $t_j^v\circ g_{\mathcal R_i}=t_i+2$ on $g^{-1}_{\mathcal R_i}(((B_{j,1}^v\cap B_{j,2'}^v)\cup (B_{j,2}^v\cap B_{j,1'}^v)\cup  (B_{j,2}^v\cap B_{j,2'}^v))\cap \{2\leq t_j^v\leq 3\})$,
\item[(v-45)] $t_j^v\circ g_{\mathcal R_i}=t_i+4$ on $g^{-1}_{\mathcal R_i}(((B_{j,1}^v\cap B_{j,1'}^v)\cup (B_{j,1}^v\cap B_{j,2'}^v)\cup B_{j,2}^v)\cap \{4\leq t_j^v\leq 5\})$,
\item[(v-W)] $t_j^v\circ g_{\mathcal R_i}=t_i+h(y_j^v)$ on $W_{i,j}=g^{-1}_{\mathcal R_i}(W_{j}^v)$.
\end{itemize}
In the above notation, $t_j^k$,  $k=e$ or $v$, is regarded as the coordinate function whose image is the 
$t_j^k$-coordinate.

Let $\mathcal R_i'$ be a collar neighborhood of $\mathcal R_i$
such that $\mathcal R_i'\times [0,1]=g^{-1}_{\mathcal R_i}(N_Q)$, and set $N_{\mathcal R_i'}=\mathcal R_i'\times [0,1]\subset \mathcal R_i\times[0,1]=N_{\mathcal R_i}$.
Choose coordinates $(r_i,\theta_i)$ on $\mathcal R_i'$ so that $\{r_i=0\}=\partial \mathcal R_i'\setminus \partial \mathcal R_i$, $\{r_i=1\}=\partial \mathcal R_i$ and the orientation of $(r_i,\theta_i)$ coincides with that of $\mathcal R_i$,
see the left in Figure~\ref{fig7}.

\begin{figure}[htbp]
\begin{center}
\includegraphics[width=11cm, bb=130 616 472 711]{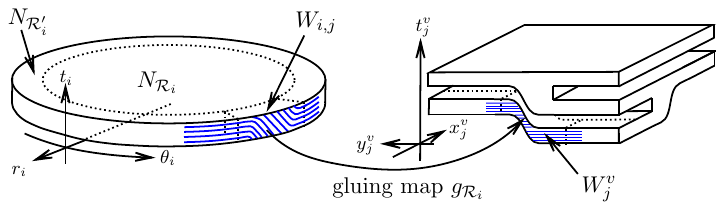}
\caption{Characteristic foliation of $\ker\eta_Q$ on $\partial \mathcal R_i\times [0,1]$.}\label{fig7}
\end{center}
\end{figure}

\begin{remark}
As described on the left in Figure~\ref{fig7},
the characteristic foliation of $\ker\eta_Q$ has negative slope as the standard contact structure on $\Real^3$.
Intuitively, this is the reason why we have the inequality $\eta\land d\eta\geq 0$ in Lemma~\ref{lemma35}. 
Remark that the figure on the right in  Figure~\ref{fig7} is a neighborhood of an $\ell$-type vertex and we cannot expect the same inequality for an $r$-type vertex. 
\end{remark}

For each $W_{i,j}$, the gluing map $g_{\mathcal R_i}|_{W_{i,j}}:W_{i,j}\to B_j^v$ is given as
\begin{equation}\label{eq30}
   (r_i,\theta_i,t_i)\mapsto (r_i,  c_i(\theta_i-c_{i,j}), t_i+h_{i,j}(\theta_i)),
\end{equation}
where $c_i$ and $c_{i,j}$ are real numbers chosen so that $c_i(\theta_i-c_{i,j})$ coincides with the $y_j^v$-coordinate on $B_j^v$, and $h_{i,j}$ is the smooth function given by $h_{i,j}(\theta_i)=h(c_i(\theta_i-c_{i,j}))$.
Note that the function $h_{i,j}$ is monotone increasing.

The gluing map $g_{\mathcal R_i}$ induces a $1$-form $\eta_{\mathcal R'_i}$ on $N_{\mathcal R'_i}$ as $\eta_{\mathcal R'_i}=g_{\mathcal R_i}^*\eta_Q$.
Due to the choice of $g_{\mathcal R_i}$, we have $\eta_{\mathcal R'_i}=dt_i$ on $N_{\mathcal R'_i}\setminus \bigcup_j W_{i,j}$,
where $j$ runs over the indices of $B_j^v$ adjacent to $\mathcal R_i$.

We define a $1$-form $\eta_{\mathcal R_i}$ on $N_{\mathcal R_i}$ as
\begin{equation}\label{eq31}
   \eta_{\mathcal R_i}=(1-\sigma(r_i))dt_i+\sigma(r_i)\eta_{\mathcal R'_i},
\end{equation}
where $\sigma(r_i)$ is a monotone increasing smooth function such that $\sigma(\varepsilon)=0$ and 
$\sigma(1-\varepsilon)=1$ for any sufficiently small $\varepsilon\geq 0$.

\begin{lemma}\label{lemmaB}
\[
    \eta_{\mathcal R_i}=
\begin{cases}
dt_i &  \text{on\;\,}N_{\mathcal R_i}\setminus\bigcup_j W_{i,j} \\
dt_i+\sigma(r_i)\frac{dh_{i,j}}{d\theta_i}(\theta_i)d\theta_i & \text{on\;\,} W_{i,j}.
\end{cases}
\]
\end{lemma}

\begin{proof}
On each $W_{i,j}$, by~\eqref{eq30}, we have $t_j^v=t_i+h_{i,j}(\theta_i)$.
Hence $\eta_{\mathcal R_i'}=g_{\mathcal R_i}^* dt_j^v=dt_i+\frac{dh_{i,j}}{d\theta_i}(\theta_i)d\theta_i$.
By~\eqref{eq31}, we have $\eta_{\mathcal R_i'}=dt_i+\sigma(r_i)\frac{dh_{i,j}}{d\theta_i}(\theta_i)d\theta_i$.
On $N_{\mathcal R_i}\setminus\bigcup_j W_{i,j}$, since
$\eta_{\mathcal R_i'}=g_{\mathcal R_i}^*dt_j^k=dt_i$ for $k=e$ and $v$, we have $\eta_{\mathcal R_i}=dt_i$ by~\eqref{eq31}. 
\end{proof}

Gluing $N_Q$ with $1$-form $\eta_Q$ and $N_{\mathcal R_i}$ with the $1$-form $\eta_{\mathcal R_i}$ for $i=1,\ldots,n$, we obtain a $1$-form $\eta_P$ on the whole $N_P$.

\begin{lemma}\label{lemma31}
\[
    d\eta_P=
\begin{cases}
0 &  \text{on\;\,}N_P\setminus \bigcup_{i,j} W_{i,j} \\
\frac{d\sigma}{dr_i}(r_i)\frac{dh_{i,j}}{d\theta_i}(\theta_i)dr_i\land d\theta_i & \text{on\;\,} W_{i,j}.
\end{cases}
\]
\end{lemma}

\begin{proof}
Since $\eta_P=dt_j^e$ on $\pr^{-1}(\Nbd(\ell_j;U_j^e))$, 
$\eta_P=dt_j^v$ on  $\pr^{-1}(\Nbd(\ell_j;U_j^v)\cup \Nbd(\ell_j';U_j^v))\setminus W_j^v$
and $\eta_P=dt_i$ on $N_{\mathcal R_i}\setminus \bigcup_j W_{i,j}$ by Lemma~\ref{lemmaB}, we have $d\eta_P=0$ on $N_P\setminus \bigcup_{i,j} W_{i,j}$.
On each $W_{i,j}$, the $2$-form in the assertion follows from Lemma~\ref{lemmaB}.
\end{proof}

\begin{lemma}\label{lemma32}
\begin{itemize}
\item[(1)] $\eta_P\land d\eta_P=0$ on $N_P\setminus \bigcup_{i,j} W_{i,j}$.
\item[(2)] $\eta_P\land d\eta_P\geq 0$ on $W_{i,j}$.
\end{itemize}
\end{lemma}

\begin{proof}
The assertion~(1) follows from Lemma~\ref{lemma31}. We prove the assertion~(2).
The function $h_{i,j}$ is monotone increasing. We also have $\frac{d\sigma}{dr_i}(r_i)\geq 0$. Hence, by Lemmas~\ref{lemmaB} and~\ref{lemma31}, 
\[
\begin{split}
\eta_{\mathcal R_i}\land d\eta_{\mathcal R_i}
=&\left(dt_i+\sigma(r_i)\frac{dh_{i,j}}{d\theta_i}(\theta_i)d\theta_i\right)\land
\left(\frac{d\sigma}{dr_i}(r_i)\frac{dh_{i,j}}{d\theta_i}(\theta_i)dr_i\land d\theta_i\right) \\
=&\frac{d\sigma}{dr_i}(r_i)\frac{dh_{i,j}}{d\theta_i}(\theta_i)dr_i\land d\theta_i\land dt_i\geq 0.
\end{split}
\]
\end{proof}

\subsection{Construction on the whole $M$}

Finally we extend the $1$-form $\eta_P$ on $N_P$ to the whole $M$.
For convenience, we set $N_D=D^2\times [-\varepsilon, 1+\varepsilon]$ for a sufficiently small $\varepsilon>0$, instead of $D^2\times [0,1]$.
Set $E=D^2\times [-\varepsilon,0]$ and $F=D^2\times [1,1+\varepsilon]$,
and denote the restrictions of $g_D$ to $E$ and $F$ by $g_E$ and $g_F$, respectively.
Set $D_\partial$ to be the union of the connected components of $B_j^k\cap g_D(N_D)$ that intersect the vertical faces of $\partial B_j^k$ corresponding to the $E$-cycle, where $B_j^k$ runs over the indices $j=1,\ldots, n_e$ for $k=e$ and $j=1,\ldots, n_v$ for $k=v$.
See the left in Figure~\ref{fig7-1}.

\begin{figure}[htbp]
\begin{center}
\includegraphics[width=15cm, bb=129 561 559 710]{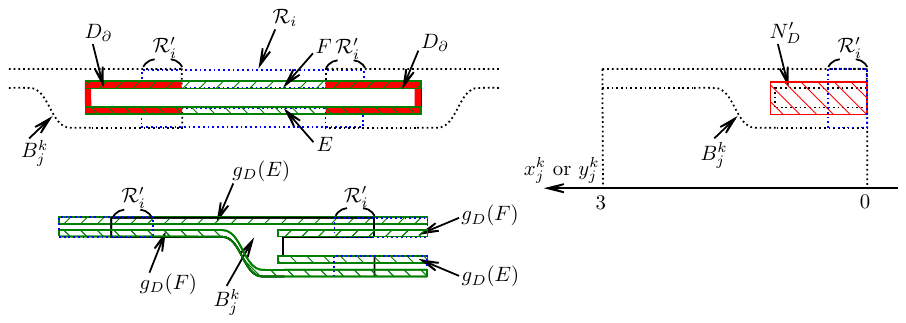}
\caption{Positions of $E$, $F$, $D_\partial$ and $N_D'$.}\label{fig7-1}
\end{center}
\end{figure}

We choose the coordinates $(u,v,w)$ of $N_D=D^2\times [-\varepsilon, 1+\varepsilon]$ 
so that they satisfy the following:
\begin{itemize}
\item[(B-D')] Set $N_D'=\{(u,v,w)\in N_D\mid (u,v,0)\in g_D^{-1}(D_\partial)\}$.
We think of $N_D'$ as being embedded in the union of the cubes $[0,L]\times [0,3]\times [0,3]$ containing $B_j^e$ and the cubes $[0,3]^2\times [0,5]$ containing $B_j^v$ so that
\[
t_j^k\circ g_D(u,v,w)=
\begin{cases}
w+1 & k=e\text{\;\,and\;\,}g_D(u,v,0)\in B_j^e \\
w+3 & k=v\text{\;\,and\;\,}g_D(u,v,0)\in H_{j,2}^v\subset B_j^v \\
w+t_j^v\circ g_D(u,v,0) & k=v\text{\;\,and\;\,}g_D(u,v,0)\in H_{j,2'}^v\subset B_j^v
\end{cases}
\]
and the first two coordinates of $B_j^k$ do not depend on the coordinate $w$.
See the right in Figure~\ref{fig7-1}.
Note that $t_j^v\circ g_D(u,v,0)=h_{i,j}(\theta_i)+1$ on the connected component of $N_D'\cap g^{-1}_D(B_j^v)$ containing $g_D^{-1}(H_{j,2'}^v)$,
where $i$ is the index $i$ of  $W_{i,j}$ and
$(r_i,\theta_i,t_i)$ are the coordinates of $g_D^{-1}(B_j^v)\cap N_{\mathcal R_i'}$.
\item[(R-E)] The coordinates $(u,v,w)$ on $E\cap  g^{-1}_D(N_{\mathcal R_i})$ are chosen so that
\[
t_i\circ g_D(u,v,w)=w+1
\]
and the first two coordinates of $N_{\mathcal R_i}=\mathcal R_i\times [0,1]$ do not depend on $w$.
\item[(R-F)] The coordinates $(u,v,w)$ on $F\cap  g^{-1}_D(N_{\mathcal R_i})$ are chosen so that 
\[
t_i\circ g_D(u,v,w)=w-1
\]
and the first two coordinates of $N_{\mathcal R_i}=\mathcal R_i\times [0,1]$ do not depend on $w$.
\item[(B-E)] The coordinates $(u,v,w)$ on the connected component of $E\cap g^{-1}_D(B_j^k)$ containing $g^{-1}_D(H_{j,4}^k)$ are chosen so that
\[
t_j^k\circ g_D(u,v,w)=w+c,
\]
where $c=3$ if $k=e$ and $c=5$ if $k=v$, and the first two coordinates of $B_j^k$ do not depend on $w$.
\item[(B-F)] The coordinates $(u,v,w)$ on the connected component of $F\cap g^{-1}_D(B_j^k)$ containing $g^{-1}_D(H_{j,1}^k)$ are chosen so that
\[
t_j^k\circ g_D(u,v,w)=w+t_j^k\circ g_D(u,v,0)
\]
and the first two coordinates of $B_j^k$ do not depend on $w$.
\end{itemize}

Due to the choice of these coordinates, we see that each of the $1$-forms $g_E^*\eta_P$ on $E$ and $g_F^*\eta_P$ on $F$ is invariant under translation to the $w$-coordinate,
which allows us to extend these forms to the whole $N_D$ canonically.
Let $\eta_E$ and $\eta_F$ denote the $1$-forms on $N_D$ obtained by extending $g_E^*\eta_P$ on $E$ and $g_F^*\eta_P$ on $F$ to $N_D$, respectively.
We then define the $1$-form $\eta_D$ on $N_D$ as
\begin{equation}\label{eq33}
   \eta_D=(1-\tau(w))\eta_E+\tau(w)\eta_F,
\end{equation}
where $\tau:[-\varepsilon,1+\varepsilon]\to [0,1]$ is a monotone increasing smooth function such that 
$\tau(w)=0$ for $-\varepsilon\leq w\leq 0$ and $\tau(w)=1$ for $1\leq w\leq 1+\varepsilon$.

By the construction, we have $g_D^*\eta_P=\eta_D$ on $E\cup F$.
Thus, for gluing the $1$-forms $\eta_D$ on $N_D$ and $\eta_P$ on $N_P$, 
it is enough to show their coincidence
on $N_D''=\{(u,v,w)\in N_D\mid g_D(u,v,1/2)\in N_Q\}$, see Figure~\ref{fig7-4}.

\begin{figure}[htbp]
\begin{center}
\includegraphics[width=8.5cm, bb=178 633 416 710]{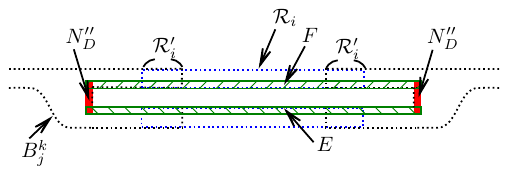}
\caption{The set $N_D''$.}\label{fig7-4}
\end{center}
\end{figure}

\begin{lemma}
$g_D^*\eta_P=\eta_D$ on $N_D''$.
\end{lemma}

\begin{proof}
Since $\eta_E=g_E^*\eta_P$ on $E\cap N_D''$ and $\eta_F=g_F^*\eta_P$ on $F\cap N_D''$ are given as the pull-backs of the same form $\eta_P=dt_j^k$ and $(u,v,w)$ satisfies the condition (B-D'), we have $\eta_E=\eta_F$ on $N_D''$.
Since $g_E^*\eta_P=\eta_D$ on $E\cap N_D''$, $g_F^*\eta_P=\eta_D$ on $F\cap N_D''$
and both of $g_D^*\eta_P$ and $\eta_D$ are defined on $N_D''$  by the linear sum of those on $E\cap N_D''$ and $F\cap N_D''$ as in~\eqref{eq33}, we have $g_D^*\eta_P=\eta_D$ on $N_D''$.
\end{proof}

Thus the $1$-forms $\eta_D$ on $N_D$ and $\eta_P$ on $N_P$ are 
glued along $g_D^{-1}(N_P)$, which we denote by $\eta$.

Topologically, the decomposition of $M$ into 
$B_j^e$ ($j=1,\ldots,n_e$), $B_j^v$ ($j=1,\ldots,n_v$),
$N_{\mathcal R_i}$ ($i=1\ldots,n$) and $N_D$ depends only on the positive flow-spine $P$.
We then choose coordinates 
$(x_j^e,y_j^e,t_j^e)$, 
$(x_j^v,y_j^v,t_j^v)$, 
$(r_i,\theta_i,t_i)$, $(u,v,w)$ on these pieces
and define the $1$-form $\eta$ by using these coordinates.

\begin{definition}
Let $\mathcal C$ be coordinates on the pieces of the decomposition of $M$ with respect to $P$. We call $\mathcal C$ a {\it coordinate system of $M$ with respect to $P$} and
the $1$-form $\eta$ obtained from $\mathcal C$ as above
the {\it reference $1$-form of $(P,\mathcal C)$}.
\end{definition}

In the rest of this section, we discuss some properties of $\eta$ and $dw$ on $N_D$ and prove the inequality $\eta\land d\eta\geq 0$.
Set $\hat W_{i,j}=\{(u,v,w)\in N_D\mid g_D(u,v,0)\in W_{i,j}\}$
and $\check W_{i,j}=\{(u,v,w)\in N_D\mid g_D(u,v,1)\in W_{i,j}\}$, see Figure~\ref{fig7-3},
and then set $\hat W=\bigcup_{i,j}\hat W_{i,j}$ and $\check W=\bigcup_{i,j}\check W_{i,j}$.

\begin{figure}[htbp]
\begin{center}
\includegraphics[width=6.5cm, bb=191 622 375 711]{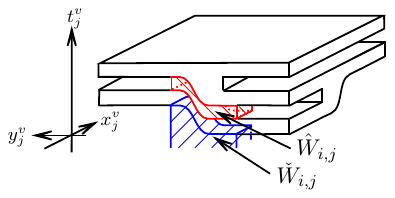}
\caption{The sets $\hat W_{i,j}$ and $\check W_{i,j}$.}\label{fig7-3}
\end{center}
\end{figure}

We first introduce the sets $S_E$ and $S_F$ on $D^2\times\{-\varepsilon\}$ and $D^2\times \{1+\varepsilon\}$, respectively. 
Roughly speaking, $S_E$ and $S_F$ are graphs of the DS-diagram of $P$ on the lower
and upper hemispheres, respectively, when we regard $N_D$ as the $3$-ball for the DS-diagram. 
For each $i=1,\ldots,n_e$, we set 
$T^e_{i,E}=B_i^e\cap  \{y_i^e=\frac{3}{2}\}\cap g_D(D^2\times\{-\varepsilon\})$
and  $T^e_{i,F}=B_i^e\cap  \{y_i^e=\frac{3}{2}\}\cap g_D(D^2\times\{1+\varepsilon\})$.
For each $j=1,\ldots,n_v$, we set 
\[
\begin{split}
T^v_{j,E}=
&B_j^v\cap  g_D(D^2\times\{-\varepsilon\})\cap \\
&\left(
\left\{y_j^v=\frac{3}{2},\; t_j^v>4\right\}
\cup 
\left\{x_j^v=\frac{3}{2},\;\frac{3}{2}\leq y_j^v\leq 3\right\}
\cup 
\left\{x_j^v=\frac{3}{2},\; t_j^v<3\right\}
\right)
\end{split}
\]
and
\[
\begin{split}
T^v_{j,F}=
&B_j^v\cap g_D(D^2\times\{1+\varepsilon\})\cap \\
&\left(
\left\{x_j^v=\frac{3}{2},\; t_j^v<4\right\}
\cup
\left\{y_j^v=\frac{3}{2},\; \frac{3}{2}\leq x_j^v\leq 3,\right\}
\cup 
\left\{y_j^v=\frac{3}{2},\; t_j^v>2\right\}
\right).
\end{split}
\]
See the left in Figure~\ref{fig7-7}, where most parts of $T^v_{j,F}$ are hidden behind.
Then $S_E$ and $S_F$ are set as
\[
   S_E=g_D^{-1}\left(\left(\bigcup_{i=1}^{n_e} T^e_{i,E}\right)\cup\left(\bigcup_{j=1}^{n_v} T^v_{j,E}\right)\right)
\text{\;\,and \;\,}
   S_F=g_D^{-1}\left(\left(\bigcup_{i=1}^{n_e} T^e_{i,F}\right)\cup\left(\bigcup_{j=1}^{n_v} T^v_{j,F}\right)\right).
\]

\begin{figure}[htbp]
\begin{center}
\includegraphics[width=14cm, bb=129 613 528 713]{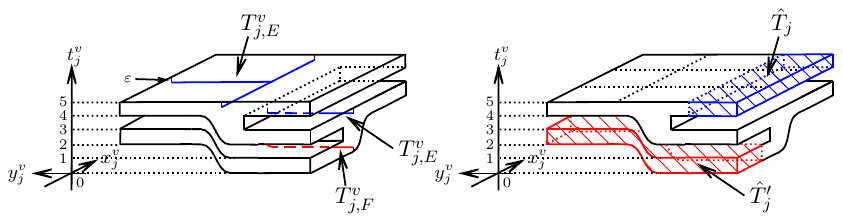}
\caption{The sets $T^v_{j,E}$, (a piece of) $T^v_{j,F}$, $\hat T_j$ and $\hat T_j'$.}\label{fig7-7}
\end{center}
\end{figure}

To explain the properties of $\eta$ and $dw$, we decompose $N_D$ into several pieces and observe them on each of these pieces.
Recall that the coordinates $(x_j^e, y_j^e, t_j^e)$ on $B_j^e$ and
 $(x_j^v, y_j^v, t_j^v)$ on $B_j^v$ are chosen, in Section~\ref{sec52},  such that
\[
\begin{split}
\pr^{-1}(\Nbd(\ell_j;U_j^e))&=\{(x_j^e, y_j^e, t_j^e)\in B_j^e\mid 1\leq y_j^e\leq 2\},\\
\pr^{-1}(\Nbd(\ell_j;U_j^v))&=\{(x_j^v, y_j^v, t_j^v)\in B_j^v\mid 1\leq y_j^v\leq 2\}, \text{\; and}\\
\pr^{-1}(\Nbd(\ell'_j;U_j^v))&=\{(x_j^v, y_j^v, t_j^v)\in B_j^v\mid 1\leq x_j^v\leq 2\}.
\end{split}
\]
Let $N_Q'\subset N_Q$ be the closure of the union of $\pr^{-1}(\Nbd(\ell_i;U_i^e))$ for $i=1,\ldots,n_e$ and
 $\pr^{-1}\left(\Nbd(\ell_j;U_j^v) \cup \Nbd(\ell_j';U_j^v)\right)\setminus (\hat T_j\cup \hat T_j')$ for $j=1,\ldots,n_v$ in $N_Q$,
where
\[
\hat T_j=
B_{j,2}^v \cap \left\{(x_j^v, y_j^v, t_j^v)\in B_j^v
 \mid 4\leq t_j^v\leq 5\right\}
\]
and 
\[
\hat T'_j=
B_{j,2'}^v \cap \left\{(x_j^v, y_j^v, t_j^v)\in B_j^v\mid 
h(y_j^v) \leq t_j^v\leq h(y_j^v) +1\right\},
\]
see the right in Figure~\ref{fig7-7}.
Then we set $A_F=g_F^{-1}(N_Q')\cap (D^2\times \{1+\varepsilon\})$
and
$N_A=(A_F\times [-\varepsilon, 1+\varepsilon])\cup \check W$.
See Figure~\ref{fig7-2},
where $\pr_{D^2}$ is the first projection $N_D=D^2\times [-\varepsilon, 1+\varepsilon]\to D^2$,

\begin{figure}[htbp]
\begin{center}
\includegraphics[width=14cm, bb=128 569 563 711]{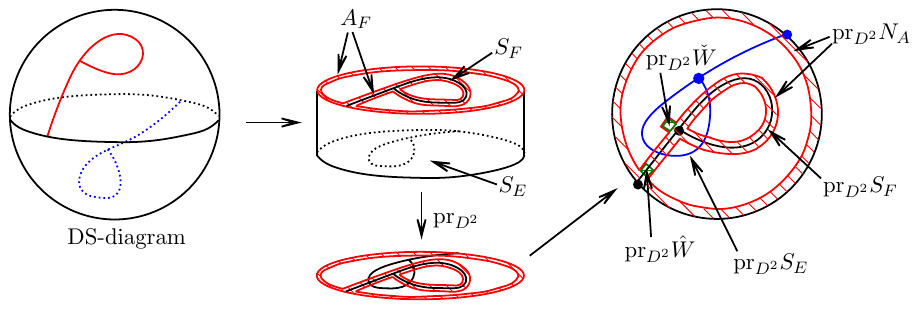}
\caption{The positions of $S_E$, $S_F$, $A_F$, $N_A$, $\hat W_{i,j}$ and $\check W_{i,j}$.}\label{fig7-2}
\end{center}
\end{figure}

The next lemma is about $\eta$ outside $\hat W\cup \check W$.

\begin{lemma}\label{lemmaC}
\begin{itemize}
\item[(1)] $\eta_D=dw$ on $N_D\setminus N_A$.
\item[(2)] $d\eta_E=d\eta_F=0$ on $N_A\setminus (\hat W\cup\check W)$.
\item[(3)] $\eta_E=dw$ on $N_A\setminus N_D'$.
\item[(4)] $\eta_E=\eta_F=dw$  on $N_D'\setminus \hat W$.
\end{itemize}
\end{lemma}

\begin{proof}
On $E\setminus N_A$,
$\eta_E$ is given by either $g^*_Ddt_i$ by Lemma~\ref{lemmaB} or $g^*_Ddt_j^k$, $k=e$ or $v$, by the construction of $\eta_E$. Hence $\eta_E=dw$ outside $N_A$ by the conditions~(B-D'), (R-E) and (B-E). Note that $\hat W\cup\check W \subset N_A$.
Next we observe $\eta_F$ on  $F\setminus N_A$. This piece is the disjoint union of $g_F^{-1}(N_{\mathcal R_i})\setminus \check W$ for $i=1,\ldots,n$, and  we have $\eta_F=g^*_Ddt_i$ on each $g_F^{-1}(N_{\mathcal R_i})\setminus \check W$.
Since $dw=g^*_Ddt_i$ by the condition~(R-F), we have $\eta_F=dw$. 
Thus, by~\eqref{eq33}, we have the assertion~(1).

Since $\eta_P$ is defined as $dt_j^k$ on 
$\left(g_E(E\cap N_A)\cap B_j^k \right)\setminus W_j^v$ and $dt_i$ on 
$\left(g_E(E\cap N_A)\cap N_{\mathcal R_i}\right)\setminus \bigcup_jW_{i,j}$, 
we have $d\eta_E=0$ on $N_A\setminus \hat W$.
We also have $d\eta_F=0$ on $N_A\setminus \check W$ by the same reason. This proves the assertion~(2).
On  $N_A\setminus N_D'$,
$\eta_E$ is given by either $g^*_Ddt_i$ by Lemma~\ref{lemmaB} or $g^*_Ddt_j^k$ by the construction of $\eta_E$. Hence $\eta_E=dw$ by the conditions~(R-E) and~(B-E).
This proves the assertion~(3). The assertion~(4) follows from the condition~(B-D').
\end{proof}

In the next two lemmas, we discuss $dw$ and $\eta$ on $\hat W$ and $\check W$.
The $1$-form $g^*_Edt_i$ on $E\cap \hat W_{i,j}$ is invariant under translation to the $w$-coordinate.
We extend it to the whole $\hat W_{i,j}$ by this invariance and denote it again 
by $g^*_Edt_i$. This is used in (1) below. The extensions in (3) below are also given by the same idea.

\begin{lemma}\label{lemmaA}
\begin{itemize}
\item[(1)] 
$dw=g^*_Edt_i$ on $\hat W_{i,j}$.
\item[(2)] $\eta_E=g_E^*\left(dt_i+\sigma(r_i)\frac{dh_{i,j}}{d\theta_i}(\theta_i)d\theta_i\right)$ on $E\cap \hat W_{i,j}$.
\item[(3)] $\eta_F=g_F^*dt_j^k=dw+g_E^*\left(\frac{dh_{i,j}}{d\theta_i}(\theta_i)d\theta_i\right)=g_E^*\left(dt_{i}+\frac{dh_{i,j}}{d\theta_{i}}(\theta_{i})d\theta_{i}\right)$ on $F\cap \hat W_{i,j}$, where $g_E^*\left(\frac{dh_{i,j}}{d\theta_i}(\theta_i)d\theta_i\right)$ and $g_E^*\left(dt_{i}+\frac{dh_{i,j}}{d\theta_{i}}(\theta_{i})d\theta_{i}\right)$ are the extensions of those on $E\cap \hat W_{i,j}$ to the whole $\hat W_{i,j}$.
\end{itemize}
\end{lemma}

\begin{proof}
Since $dt_j^k=dw+\frac{dh_{i,j}}{d\theta_i}(\theta_i)d\theta_i$ by the condition~(B-D') and $dt_j^k=dt_i+\frac{dh_{i,j}}{d\theta_i}(\theta_i)d\theta_i$ by the definition of $g_{\mathcal R_i}$ on $W_{i,j}$,
we have $dw=g^*_Edt_i$ on $E\cap \hat W_{i,j}$. 
Since both of $dw$ and $g^*_Edt_i$ are invariant under translation to the $w$-coordinate,
the assertion~(1) holds.
On $E\cap \hat W_{i,j}$, we have the assertion~(2) by Lemma~\ref{lemmaB}.
On $F\cap \hat W_{i,j}$, $g^*_Fdt_j^k
=dw+g^*_F\left(\frac{dh_{i,j}}{d\theta_i}(\theta_i)d\theta_i\right)
=dw+g^*_E\left(\frac{dh_{i,j}}{d\theta_i}(\theta_i)d\theta_i\right)$ 
by the condition~(B-D')
and $g^*_Fdt_j^k=g_E^*dt_j^k=g^*_E\left(dt_i+\frac{dh_{i,j}}{d\theta_i}(\theta_i)d\theta_i\right)$ by the condition~(B-D') and the gluing map~\eqref{eq30}.
Thus the assertion~(3) holds.
\end{proof}

The $1$-form $g^*_Fdt_i$ on $F\cap \check W_{i,j}$ is invariant under translation to the $w$-coordinate.
We extend it to the whole $\check W_{i,j}$ by this invariance and denote it by again 
by $g^*_Fdt_i$. 

\begin{lemma}\label{lemmaD}
\begin{itemize}
\item[(1)] 
$dw=g^*_Fdt_i$ on $\check W_{i,j}$.
\item[(2)] $\eta_E=dw$ on $E\cap \check W_{i,j}$.
\item[(3)] $\eta_F=g_F^*\left(dt_i+\sigma(r_i)\frac{dh_{i,j}}{d\theta_i}(\theta_i)d\theta_i\right)$ on $F\cap \check W_{i,j}$.
\end{itemize}
\end{lemma}

\begin{proof}
Since $dt_j^k=dw+\frac{dh_{i,j}}{d\theta_i}(\theta_i)d\theta_i$ by the condition~(B-F)
and $dt_j^k=dt_i+\frac{dh_{i,j}}{d\theta_i}(\theta_i)d\theta_i$ by the gluing map~\eqref{eq30},
we have $dw=g^*_Fdt_i$ on $F\cap \check W_{i,j}$. 
Since both of $dw$ and $g^*_Fdt_i$ are invariant under translation to the $w$-coordinate,
the assertion~(1) holds.
The assertion~(2) follows from the property that $\hat W\cap \check W=\emptyset$.
The assertion~(3) follows from Lemma~\ref{lemmaB}.
\end{proof}

The next inequality is important for making contact forms in the proof of Theorem~\ref{thm01}.

\begin{lemma}\label{lemma35}
$\eta\land d\eta\geq 0$ on $M$.
\end{lemma}

\begin{proof}
The inequality holds on $M\setminus N_D$ by Lemma~\ref{lemma32}.
On $N_D\setminus N_A$, $\eta\land d\eta=0$ by Lemma~\ref{lemmaC}~(1). 
From~\eqref{eq33}, we have
\[
\begin{split}
   \eta_D\land d\eta_D
=&
\left((1-\tau(w))\eta_E+\tau(w)\eta_F\right) 
\land \left(-\frac{d\tau}{dw}(w)dw\land \eta_E\right.\\
&\left.+(1-\tau(w))d\eta_E+\frac{d\tau}{dw}(w)dw\land \eta_F+\tau(w)d\eta_F\right) \\
=&
(1-\tau(w))^2\eta_E\land d\eta_E
+(1-\tau(w))\frac{d\tau}{dw}(w)\eta_E\land dw\land \eta_F \\
&+(1-\tau(w))\tau(w)\eta_E \land d\eta_F 
 -\tau(w)\frac{d\tau}{dw}(w)\eta_F\land dw\land \eta_E \\
&+(1-\tau(w))\tau(w)\eta_F \land d\eta_E 
+\tau(w)^2\eta_F\land d\eta_F \\
=&
(1-\tau(w))^2\eta_E\land d\eta_E
+\frac{d\tau}{dw}(w)\eta_E\land dw\land \eta_F \\
&+(1-\tau(w))\tau(w)(\eta_E \land d\eta_F+\eta_F \land d\eta_E) 
+\tau(w)^2\eta_F\land d\eta_F.
\end{split}
\]
This form vanishes on $N_A\setminus (\hat W\cup \check W)$ by Lemma~\ref{lemmaC}.
Thus it is enough to show the inequality on $\hat W\cup \check W$.

Let $(g_E^*dt)_{\hat W_{i,j}}$ denote the extension of $g_E^*dt_i$ on $E\cap \hat W_{i,j}$ to $\hat W_{i,j}$.
By Lemma~\ref{lemmaA}~(2) and~(3), we have $\eta_E\land (g_E^*dt)_{\hat W_{i,j}}\land \eta_F=0$ on $\hat W_{i,j}$.
By Lemma~\ref{lemmaA}~(1), $dw=(g_E^*dt)_{\hat W_{i,j}}$.
Hence the term $\frac{d\tau}{dw}(w)\eta_E\land dw\land \eta_F$ vanishes on $\hat W_{i,j}$.
The terms with $d\eta_F$ vanish since $d\eta_F=0$ by~Lemma~\ref{lemmaA}~(3).
The $3$-form $\eta_E\land d\eta_E$ is calculated, by Lemma~\ref{lemmaA}~(2), as
\[
\begin{split}
   \eta_E\land d\eta_E
&=g_E^*\left(dt_i+\sigma(r_i)\frac{dh_{i,j}}{d\theta_i}(\theta_i)d\theta_i\right)\land 
g_E^*\left(
\frac{d\sigma}{dr_i}(r_i)\frac{dh_{i,j}}{d\theta_i} (\theta_i)dr_i\land d\theta_i\right) \\ 
&=g_E^*\left(\frac{d\sigma}{dr_i}(r_i)\frac{dh_{i,j}}{d\theta_i}(\theta_i)dr_i\land d\theta_i\land dt_i\right),
\end{split}
\]
where the pull-backs of forms by $g_E$ are regarded as their extensions to $N_D$.
Also, by Lemma~\ref{lemmaA}~(2) and~(3), we have
\[
\begin{split}
   \eta_F\land d\eta_E
&=g_E^*\left(dt_{i}+\frac{dh_{i,j}}{d\theta_i}(\theta_i)d\theta_{i}\right)\land 
g_E^*\left(
\frac{d\sigma}{dr_i}(r_i)\frac{dh_{i,j}}{d\theta_i} (\theta_i)dr_i\land d\theta_i\right) \\ 
&=g_E^*\left(\frac{d\sigma}{dr_i}(r_i)\frac{dh_{i,j}}{d\theta_i}(\theta_i)dr_i\land d\theta_i\land dt_i\right),
\end{split}
\]
where the pull-backs of forms by $g_E$ are regarded as their extensions to $N_D$.
Since $\sigma$ and $h_{i,j}$ are monotone increasing, these two terms are non-negative.
Thus $\eta_D\land d\eta_D\geq 0$ on $\hat W$.

On $\check W_{i,j}$, by Lemma~\ref{lemmaD}~(2),
\[
   d\eta_D=d((1-\tau(w))\eta_E+\tau(w)\eta_F)
=\frac{d\tau}{dw}dw\land \eta_F+\tau(w)d\eta_F
\]
and hence, by Lemma~\ref{lemmaD}, we have
\[
\begin{split}
   \eta_D\land d\eta_D&=(1-\tau(w))\tau(w)dw\land d\eta_F
+\tau(w)^2\eta_F\land d\eta_F \\
&=\tau(w)g_F^*\left(\frac{d\sigma}{dr_i}(r_i)\frac{dh_{i,j}}{d\theta_i}(\theta_i)dr_i\land d\theta_i\land dt_i\right) \geq 0,
\end{split}
\]
where the pull-back of a form by $g_F$ is regarded as its extension to $N_D$.
This completes the proof.
\end{proof}

\section{Existence}

We use the same notations as in the previous section.
In this section, we prove the following theorem.

\begin{theorem}\label{thm1}
Let $P$ be a positive flow-spine of a closed, connected, oriented, smooth $3$-manifold $M$.
Then there exists a contact form on $M$ whose Reeb flow is carried by $P$.
\end{theorem}

\subsection{The $1$-form $\beta$ on $P$}\label{sec61}

In this subsection, we construct a $1$-form $\beta$ on a positive flow-spine $P$ of $M$ that
will be used to give a contact form on $M$.

Let $P$ be a positive flow-spine of $M$.
We assume that $P$ is embedded in $M$ so that, for each edge $e$ of $P$,
the closure of the union of any pair of adjacent regions that induce opposite orientations to $e$ is smooth along $e$.
Hereafter, a $1$-form on $P$ means that it is obtained by gluing the pullback of a $1$-form on $Q'$ 
by the projection $\pr_Q:Q\to Q'$ and $1$-forms on the smooth surfaces $R_1,\ldots,R_n$.

Set $G_E=\pr_{D^2}(S_E)$ and $G_F=\pr_{D^2}(S_F)$.
Each region $R_i$ of $P$ corresponds to a region $S_i$ on $D^2$ bounded by $G_E\cup\partial D^2$ as 
$S_i=\pr_{D^2}\circ g_D^{-1}(R_i\times\{1\})$.
Each region $S_i$ divides into several regions by $G_F$.
For example, the $1$-gon bounded by $G_E$ in Figure~\ref{fig7-6} divides into three regions by $G_F$. 
Let $G_{R_i}$ denote the graph 
on $R_i$ obtained as the image of $(g_E\circ\pr_{D^2}^{-1})(G_F\cap S_i)$ by 
the first projection $\pr_{\mathcal R_i}:N_{\mathcal R_i}=\mathcal R_i\times [0,1]\to \mathcal R_i$ 
with extending each endpoint on $\pr_{\mathcal R_i}(\partial \mathcal R_i)$ to $\partial R_i$. 
We denote the neighborhood $(\pr_{\mathcal R_i}\circ g_E(N_A))\cup \pr(N_Q'\cap R_i)$ of $G_{R_i}\cup\partial R_i$ in $R_i$ by $A_{R_i}$ and the connected components of $R_i\setminus G_{R_i}$ by $R_{i1},\ldots,R_{in_i}$,
see Figure~\ref{fig7-6}.

\begin{figure}[htbp]
\begin{center}
\includegraphics[width=15cm, bb=130 585 544 710]{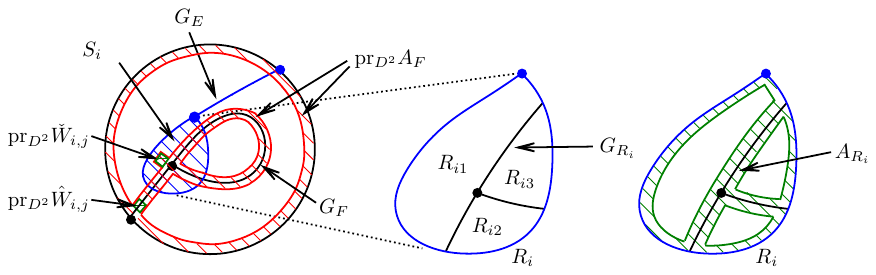}
\caption{The graph $G_{R_i}$, the divided regions $R_{i1},\ldots,R_{in_i}$ and the neighborhood $A_{R_i}$.}\label{fig7-6}
\end{center}
\end{figure}

We may choose a coordinate system $\mathcal C$ of $M$ with respect to $P$ so as 
to satisfy the conditions~(B-D'), (R-E), (R-F), (B-E) and (B-F) in the previous section
and the following additional condition:
\begin{itemize}
\item[(A)] There exists a deformation retraction from $A_{R_i}$ to $G_{R_i}\cup\partial R_i$ in each $R_i$.
\end{itemize}

\begin{definition}
A coordinate system $\mathcal C$ is said to be {\it non-degenerate} if 
it satisfies the conditions~(B-D'), (R-E), (R-F), (B-E), (B-F) and~(A).
\end{definition}

For a $1$-form $\beta$ on $P$,
let $\hat\beta$ be the $1$-form on $N_P$ defined by $\pr^*\beta$ on $N_Q$ and 
\begin{equation}\label{eq41}
   (1-\sigma(r_i))(\pr\circ g_{\mathcal R_i})^*\beta+\sigma(r_i)\pr_{\mathcal R_i}^*\beta
\end{equation}
on $N_{\mathcal R_i}$.
The $1$-forms $g_E^*\hat\beta$ on $E$ and $g_F^*\hat\beta$ on $F$ can extend to the whole $N_D$
since they are invariant under translation to the $w$-coordinate.
Let $\hat\beta_E$ and $\hat\beta_F$ denote the $1$-forms on $N_D$ obtained by extending 
$g_E^*\hat\beta$ on $E$ and $g_F^*\hat\beta$ on $F$ to $N_D$, respectively.

\begin{lemma}\label{lemmaE}
The $1$-form $\hat\beta$ does not depend on the variable $t_j^k$ on $B_j^k$ for $k=e, v$
and the variable $t_i$ on $N_{\mathcal R_i}$.
\end{lemma}

\begin{proof}
On $N_P\setminus \bigcup_{i,j} W_{i,j}$, the assertion is obvious from the construction of $\hat\beta$.
On $W_{i,j}$, both of $(\pr\circ g_{\mathcal R_i})^*\beta$ 
and $\pr_{\mathcal R_i}^*\beta$ depend only on the coordinates $(r_i,\theta_i)$ by~\eqref{eq41}. Thus the assertion follows.
\end{proof}

Let $\eta$ be the reference $1$-form of $(P,\mathcal C)$, where $\mathcal C$ is a non-degenerate coordinate system.
Let $\eta_E$ and $\eta_F$ be the $1$-forms on $N_D$ used when we define $\eta$ in the previous section.

\begin{lemma}\label{lemma41}
There exists a $1$-form $\beta$ on $P$ such that
\begin{itemize}
\item[(a)] $d\beta>0$ on $P$, and
\item[(b)] $\max\{|B_E|, |B_F|\}<\frac{1}{2}\min\{d\hat\beta_E\land \eta,
d\hat\beta_F\land\eta\}$ on $N_A$,
\end{itemize}
where 
\[
\begin{split}
   B_E&=\hat\beta_E\land\left(-\frac{d\tau}{dw}(w)\eta\land dw+(1-\tau(w))d\eta\right)
\text{\;\,and} \\
   B_F&=\hat\beta_F\land\left(\frac{d\tau}{dw}(w)\eta\land dw+\tau(w)d\eta\right).
\end{split}
\]
\end{lemma}

\begin{proof}
Since $P$ is a positive flow-spine, it satisfies the admissibility condition by Proposition~\ref{prop01}. Hence we can choose real numbers $x_1,\ldots,x_m$
that satisfy inequality~\eqref{eq1000}.
Each region $R_i$ divides into several regions $R_{ij}$ by $G_{R_i}$.
If an edge $e$ of $S(P)$ with a real number $x(e)$ divides into $k_e$ edges,
then we assign $x(e)/k_e$ to each of the $k_e$ edges, where the orientations of the $k_e$ edges
are chosen so that they are consistent with that of $e$.
We then assign orientations to the edges of $G_{R_i}$ and assign real numbers to them
so that, for each region $R_{ij}$, the sum of all assigned real numbers along $\partial R_{ij}$ is strictly positive.
The existence of such an assignment will be proved in Lemma~\ref{lemma42} below.

Now we take the real $2$-plane $\Real^2$ with the coordinates $(x,y)$ and set the $1$-form $x^sdy$ on it, where $s$ is a positive real number.
Set $U=\{(x,y)\in\Real^2\mid 0<x<1\}$.
For each vertex $v$ of $\pr(S(P))$, we choose an open disk $Q_v$ in $U$ with the $1$-form $x^sdy$.
For each edge $e$ of $\pr(S(P))$ connecting vertices $v_1$ and $v_2$,
we choose an embedded arc $\gamma_e$ in $U$ with an open neighborhood $Q_e$ in $U$ satisfying that 
\begin{itemize}
\item[(1)] the $1$-form on $Q_e$ near the endpoints coincides with those on $Q_{v_1}$ and $Q_{v_2}$, and
\item[(2)] $\int_{\gamma_e} x^sdy=x(e)$, where $x(e)$ is the real number assigned to the edge $e$.
\end{itemize}
We call the neighborhood $Q_e$ an {\it open band}.
We can find such an open band 
since the $1$-form $x^sdy$ does not depend on the $y$-coordinate while the integrated value depends on the $y$-coordinate.
By gluing such $1$-forms for all open disks $\{Q_v\}$ and open bands $\{Q_e\}$
and identifying the union of these pieces with $Q'$ so that the arc in each open band is sent to the corresponding edge of $S(P)$ in $Q'$, we obtain a $1$-form $\beta_{Q'}$ on $Q'$.

We then define a $1$-form $\beta_Q$ on $Q$ by $\pr_Q^*\beta_{Q'}$, where $\pr_Q:Q\to Q'$ is the projection introduced in Section~\ref{sec52}.

On each $R_i$, 
we set the $1$-form on $g_{R_i}^{-1}(Q)$ by $g_{R_i}^*\beta_Q$,
where $g_{R_i}:\mathcal R_i'\times \{\frac{1}{2}\}\to Q$ is the restriction of $g_{\mathcal R_i}$ to $\mathcal R_i'\times \{\frac{1}{2}\}$.
Note that $g_{R_i}^*\beta_Q$ is induced from the form $x^sdy$ on $U$ for each of neighborhoods of vertices and edges of $S(P)$ in $\partial R_i$.
We then extend it to a $1$-form $\beta_{A_{R_i}}$ on $A_{R_i}$ using $x^sdy$ by the same manner.
Next, we extend the $1$-form $\beta_{A_{R_i}}$ to the whole $R_i$ by using the argument of
Thurston and Winkelnkemper~\cite{TW75}.
For each region $R_{ij}$, choose a volume form $\Omega_{ij}$ on $R_{ij}$ that satisfies 
$\int_{R_{ij}}\Omega_{ij}=\int_{\partial R_{ij}}\beta_{A_{R_i}}$ and
$\Omega_{ij}=d(\beta_{A_{R_i}})$ on $A_{R_i}\cap R_{ij}$. 
Let $\beta'_{R_{ij}}$ be any $1$-form on $R_{ij}$ that equals $\beta_{A_{R_i}}$ on $A_{R_i}\cap R_{ij}$.
By Stokes' theorem, we have
\[
  \int_{R_{ij}}(\Omega_{ij}-d\beta_{R_{ij}}')=\int_{R_{ij}}\Omega_{ij}-\int_{\partial R_{ij}}\beta_{R_{ij}}'
  =\int_{R_{ij}}\Omega_{ij}-\int_{\partial R_{ij}}\beta_{A_{R_i}}=0.
\]
The closed $2$-form $\Omega_{ij}-d\beta_{R_{ij}}'$ represents the trivial class in cohomology vanishing on $A_{R_i}\cap R_{ij}$. 
By de Rham's theorem, there is a $1$-form $\zeta_{i,j}$ on $R_{ij}$ vanishing on $A_{R_i}\cap R_{ij}$ and satisfying $d\zeta_{ij}=\Omega_{ij}-d\beta_{R_{ij}}'$. 
Define $\beta_{R_{ij}}=\beta_{R_{ij}}'+\zeta_{ij}$.
Then $d\beta_{R_{ij}}=\Omega_{ij}$ is a volume form on $R_{ij}$ and $\beta_{R_{ij}}$ 
coincides with $\beta_{A_{R_i}}$ on $A_{R_i}\cap R_{ij}$. 

Gluing  $\beta_{A_{R_i}}$ on $A_{R_i}$ and $\beta_{R_{ij}}$ on $R_{ij}$'s for each $R_i$ and then gluing them with the $1$-forms $\beta_Q$ on $Q$, we obtain a $1$-form $\beta$ on $P$ that satisfies $d\beta>0$ on $P$.

Next we check the condition~(b).
By the construction,  $\hat \beta_E$ and $\hat\beta_F$ on $N_A$
are defined by gluing $1$-forms of type $x^sdy$.
The convergence of $x^sdy$ as $s$ goes to the infinity is faster than that of $d(x^sdy)=sx^{s-1}dx\land dy$ since $0<x<1$.
Hence the inequality on $N_A$ in the condition~(b) is satisfied for a sufficiently large $s$ if the right-hand side is positive.
On $N_A$, the forms $\eta_E$ and $\eta_F$ are either $dt_j^k$ for $k=e$ or $v$, $dt_i$ or those
given in Lemmas~\ref{lemmaA} and~\ref{lemmaD}.
Since $\eta$ is given by the linear sum of $\eta_E$ and $\eta_F$ as in~\eqref{eq33},
Lemma~\ref{lemmaE} and the condition~(a) imply the inequalities $d\hat\beta_E\land\eta>0$ and $d\hat \beta_F\land\eta>0$. This completes the proof.
\end{proof}

For an immersed oriented circle or arc $a$ and an edge $e$ in an oriented graph, set $I(a,e)$ to be the integer that counts how many times $a$ passes through $e$ algebraically.

\begin{lemma}\label{lemma42}
Suppose that real numbers $x_1,\ldots,x_{n_i}$ are assigned to the oriented edges $e_1,\ldots,e_{n_i}$ on $\partial R_i$ such that 
$\sum_{k=1}^{n_i}I(\partial R_i,e_k)x_k>0$. 
Then there exists an assignment of real numbers 
$x'_1,\ldots,x'_{n_i'}$ to the oriented edges $e'_1,\ldots,e'_{n_i'}$ of $G_{R_i}$ such that they satisfy $\sum_{k=1}^{n_i'}I(\partial R_{ij},e'_k)x'_k>0$ for each region $R_{ij}$.
\end{lemma}

\begin{proof}
Adding an edge to $G_{R_i}$ if necessary, we may assume that $G_{R_i}\cup\partial R_i$ is connected.
Then, the graph $G_{R_i}$ is obtained from $\partial R_i$ by adding 
\begin{itemize}
\item[(I)] a vertex to the interior of an edge,
\item[(II)] an edge to a vertex (in consequence, a new vertex is added to the other end of the edge),
\item[(III)] an edge connecting two vertices, where the two vertices are possibly the same vertex.
\end{itemize}

We prove the assertion by induction
on the total number of steps (I), (II) and (III) to obtain $G_{R_i}\cup \partial R_i$ from $\partial R_i$.
The assertion is true for the graph on $\partial R_i$ by the assumption.
Suppose that the assertion holds after applying the steps (I), (II) and (III) finite times and consider to apply one of them once more. 
In case (I), if the real number assigned to the edge is $x$ then we assign $x/2$ to each of the two edges separated by the new vertex, where the orientations of these two edges are the one induced from the original edge. This assignment satisfies the assertion.
In case (II), we can assign any real number to the edge since both sides of the added edge belong to the same region.
Hence this assignment satisfies the assertion by the assumption of the induction.

In case (III), the two regions adjacent to the added edge
are different. Denote these two regions by $R'_1$ and $R'_2$. For $k=1,2$, let $x'_k$ be the sum of assigned real numbers along the boundary of $R'_k$ except the added edge.
By the assumption of the induction, we have $x'_1+x'_2>0$.
Set $\delta=(x'_1+x'_2)/2$ and assign the real number $\delta-x'_1$ to the added edge with orientation induced by $R'_1$. Then the sum of real numbers along the boundary of $R'_1$ is $\delta>0$ and that along the boundary of $R'_2$ is 
$x'_2-(\delta-x'_1)=\delta>0$. Thus the assertion holds, and this completes the proof by induction.
\end{proof}

\subsection{Proof of Theorem~\ref{thm1}}

Let $\mathcal C$ be a non-degenerate coordinate system of $P$ and 
$\eta$ be the reference $1$-form of $(P,\mathcal C)$.
Let $\beta$ be a $1$-form on $P$ constructed in Lemma~\ref{lemma41}.
The $1$-form $\hat\beta$ on $N_P$ is defined as explained before Lemma~\ref{lemmaE}.

\begin{lemma}\label{lemma42a}
$\hat\beta\land d\hat\beta=0$, $\hat\beta\land d\eta=0$ and $d\hat\beta\land \eta>0$ on $N_P$.
\end{lemma}

\begin{proof}
Lemma~\ref{lemmaE} implies $\hat\beta\land d\hat\beta=0$,
Lemma~\ref{lemma31} implies $\hat\beta\land d\eta=0$ on $N_P\setminus\bigcup_{i,j} W_{i,j}$,
and Lemmas~\ref{lemmaE} and~\ref{lemma31} imply $\hat\beta\land d\eta=0$ on $W_{i,j}$.
The last assertion follows from the linear sum \eqref{eq31} with
$\eta_{\mathcal R_i'}=g^*_{\mathcal R_i}(dt_j^k)$ and Lemma~\ref{lemma41}~(a). 
\end{proof}

\begin{lemma}\label{lemma43}
The $1$-form $\alpha_{P,R}=\hat \beta+R\eta$ is a contact form on $N_P$ whose Reeb vector field is positively transverse to $P$ 
for any $R>0$.
\end{lemma}

\begin{proof}
The $1$-form $\alpha_{P,R}$ is a contact form on $B^k_j\setminus \bigcup_i g_{\mathcal R_i}(N_{\mathcal R_i})$, 
for $k=e$ or $v$, 
since the form is given as $\alpha_{P,R}=\hat\beta+Rdt^k_j$
and $d\hat\beta\land dt_j^k>0$ holds by Lemma~\ref{lemma41}~(a).
It is also a contact form on $N_{\mathcal R_i}\setminus\bigcup_jW_{i,j}$ since the form is given as
$\alpha_{P,R}=\hat\beta+Rdt_i$ by Lemma~\ref{lemmaB} and $d\hat\beta\land dt_i>0$ by  Lemma~\ref{lemma41}~(a).
On $W_{i,j}$, the $3$-form $\alpha_{P,R}\land d\alpha_{P,R}$ is given as
\[
\alpha_{P,R}\land d\alpha_{P,R}= \hat\beta\land d\hat\beta+R\eta\land d\hat\beta+R\hat\beta\land d\eta+R^2\eta\land d\eta,
\]
which is positive by Lemma~\ref{lemma42a} and $\eta\land d\eta\geq 0$ in Lemma~\ref{lemma32}.

Next we check that 
the Reeb vector field $R_{\alpha_{P,R}}$ of $\alpha_{P,R}$ is transverse to $P$.
On $B_j^k\setminus \bigcup_i g_{\mathcal R_i}(N_{\mathcal R_i})$, 
since $\beta$ is defined from the $1$-form $x^sdy$ on $\Real^2$, 
the Reeb vector field of $\alpha_{P,R}$ is given as $\frac{\partial}{\partial t_j^k}$.
Hence the assertion holds.

On $N_{\mathcal R_i}\setminus \bigcup_{j}W_{i,j}$, $\alpha_{P,R}=\hat\beta+Rdt_i$ by Lemma~\ref{lemmaB} and hence $R_{\alpha_{P,R}}=\frac{1}{R}\frac{\partial}{\partial t_i}$ by Lemma~\ref{lemmaE}.
Thus the assertion holds.

On $W_{i,j}$, by Lemma~\ref{lemma31}, 
\[
\begin{split}
  \alpha_{P,R}\land d\alpha_{P,R}\left(R_{\alpha_{P,R}},\frac{\partial}{\partial r_i},\frac{\partial}{\partial \theta_i}\right) 
  &=\left(d\hat\beta+R\frac{d\sigma}{dr_i}(r_i)\frac{dh_{i,j}}{d\theta_i}(\theta_i)dr_i\land d\theta_i \right)
\left(\frac{\partial}{\partial r_i},\frac{\partial}{\partial \theta_i}\right) \\
  &=d\hat\beta\left(\frac{\partial}{\partial r_i},\frac{\partial}{\partial \theta_i}\right)
   +R\frac{d\sigma}{dr_i}(r_i)\frac{dh_{i,j}}{d\theta_i}(\theta_i).
\end{split}
\]
The right-hand side is positive since $d\beta>0$ by  Lemma~\ref{lemma41}~(a) and $\sigma$ and $h_{i,j}$ are monotone increasing.
Since $\alpha_{P,R}$ is a contact form, $R_{\alpha_{P,R}}$ is positive with respect to 
$\left\langle\frac{\partial}{\partial r_i},\frac{\partial}{\partial \theta_i}\right\rangle$.
This means that $R_{\alpha_{P,R}}$ is positively transverse to $P$ on $W_{i,j}$.
\end{proof}

\begin{proof}[Proof of Theorem~\ref{thm1}]
Let $\hat\beta_E$ and $\hat\beta_F$ be the $1$-forms on $N_D$ obtained from 
the $1$-form $g_D^*\hat\beta$ on $E$ and $F$, respectively, by extending it to the whole $N_D$ as we introduced before Lemma~\ref{lemmaE}.
Define the $1$-form $\alpha_R$ on $N_D=D^2\times [0,1]$ as
\begin{equation}\label{eq42}
   \alpha_R=(1-\tau(w))\hat\beta_E+\tau(w)\hat\beta_F+R\eta,
\end{equation}
where $R>0$.
This coincides with $\alpha_{P,R}$ on $g_D^{-1}(N_P)$.
To show that $\alpha_R$ is a contact form on $M$, by Lemma~\ref{lemma43},
it is enough to show that $\alpha_R$ is a contact form on $N_D$.
The $3$-form $\alpha_R\land d\alpha_R$ is calculated as
\[
\begin{split}
   \alpha_R\land d\alpha_R
=&-\frac{d\tau}{dw}(w)\hat\beta_E\land \hat\beta_F\land dw
+R\frac{d\tau}{dw}(w)(-\hat\beta_E+\hat\beta_F)\land \eta \land dw \\
&+R((1-\tau(w))\hat\beta_E+\tau(w)\hat\beta_F)\land d\eta \\
&+R((1-\tau(w)) d\hat\beta_E + \tau(w) d\hat\beta_F)\land\eta +R^2\eta\land d\eta,
\end{split}
\]
where we used $\hat\beta_E\land d\hat\beta_E=\hat\beta_F\land d\hat\beta_F=\hat\beta_E\land d\hat\beta_F=\hat\beta_F\land d\hat\beta_E=0$
since $\hat\beta_E$ and $\hat\beta_F$ depend only on the coordinates $(x,y)$ of each local coordinate chart on $\pr N_A$ and only on $(r_i,\theta_i)$ on each $N_{\mathcal R_i}$ as in Lemma~\ref{lemmaE}.

On $N_A$, since $\beta$ satisfies the condition~(b) in Lemma~\ref{lemma41},
the absolute value of the sum of the second and third terms divided by $R$ is bounded above as
\[
\begin{split}
   |B_E+B_F|&\leq |B_E|+|B_F|<\min\left\{d\hat\beta_E\land \eta,
\;\,d\hat\beta_F\land \eta
\right\} \\
   &\leq((1-\tau(w)) d\hat\beta_E + \tau(w) d\hat\beta_F)\land\eta.
\end{split}
\]
This inequality and Lemma~\ref{lemma35} show that
$\alpha_R\land d\alpha_R$ is positive on $N_A$ for a sufficiently large $R$. 
On $N_D\setminus N_A$, we have $d\eta=0$ and $dw\land \eta=0$ by Lemma~\ref{lemmaC}~(1).
We also have $d\hat\beta_E\land \eta>0$ and $d\hat\beta_F\land \eta>0$ by Lemma~\ref{lemmaC}~(1) and Lemma~\ref{lemma41}~(a)
and $\eta\land d\eta\geq 0$ by Lemma~\ref{lemma35}.
Thus $\alpha_R\land d\alpha_R$ is positive on $N_D\setminus N_A$ for a sufficiently large $R>0$.
Thus $\alpha_R$ is a contact form on $N_D$ for a sufficiently large $R$.

Next we check that the Reeb flow of $\alpha_R$ is carried by $P$.
It had been checked on $N_P$ in Lemma~\ref{lemma43}.
Let $R_{\alpha_R}$ be the Reeb vector field of $\alpha_R$.
On $N_D$, with coordinates $(u,v,w)$ of $N_D$,
\[
\begin{split}
\alpha_R\land d\alpha_R\left(R_{\alpha_R}, \frac{\partial}{\partial u}, \frac{\partial}{\partial v}\right)
=&d((1-\tau(w))\hat\beta_E+\tau(w)\hat\beta_F+R\eta)\left(\frac{\partial}{\partial u}, \frac{\partial}{\partial v}\right)\\
=&((1-\tau(w))d\hat\beta_E+\tau(w) d\hat\beta_F\\
&+R((1-\tau(w))d\eta_E+\tau(w) d\eta_F))\left(\frac{\partial}{\partial u}, \frac{\partial}{\partial v}\right),
\end{split}
\]
where we used $dw(\frac{\partial}{\partial u})=dw(\frac{\partial}{\partial v})=0$.
By Lemma~\ref{lemma41}~(a), we have $((1-\tau(w))d\hat\beta_E+\tau(w) d\hat\beta_F)\left(\frac{\partial}{\partial u}, \frac{\partial}{\partial v}\right)>0$.

On  $N_D\setminus (\hat W\cup \check W)$, the term $((1-\tau(w))d\eta_E+\tau(w) d\eta_F))\left(\frac{\partial}{\partial u}, \frac{\partial}{\partial v}\right)$ vanishes by Lemma~\ref{lemmaC}~(1) and~(2) and hence
$\alpha_R\land d\alpha_R\left(R_{\alpha_R}, \frac{\partial}{\partial u}, \frac{\partial}{\partial v}\right)>0$.

On $\hat W_{i,j}$, 
we have $((1-\tau(w))d\eta_E+\tau(w) d\eta_F))\left(\frac{\partial}{\partial u}, \frac{\partial}{\partial v}\right)\geq 0$
since 
\[
d\eta_E\left(\frac{\partial}{\partial u}, \frac{\partial}{\partial v}\right)=
g_E^*\left(\frac{d\sigma}{dr_i}(r_i)\frac{dh_{i,j}}{d\theta_i}(\theta_i)dr_i\land d\theta_i\right)\left(\frac{\partial}{\partial u}, \frac{\partial}{\partial v}\right)\geq 0
\]
on $E\cap \hat W_{i,j}$ by Lemma~\ref{lemmaA}~(2)
and $d\eta_F\left(\frac{\partial}{\partial u}, \frac{\partial}{\partial v}\right)=0$ on $F\cap \hat W_{i,j}$ by Lemma~\ref{lemmaA}~(3).

On $\check W_{i,j}$, 
we also have $((1-\tau(w))d\eta_E+\tau(w) d\eta_F))\left(\frac{\partial}{\partial u}, \frac{\partial}{\partial v}\right)\geq 0$
since $d\eta_E\left(\frac{\partial}{\partial u}, \frac{\partial}{\partial v}\right)=0$ on $E\cap \check W_{i,j}$ by Lemma~\ref{lemmaD}~(2) and 
\[
d\eta_F\left(\frac{\partial}{\partial u}, \frac{\partial}{\partial v}\right)=
g_F^*\left(\frac{d\sigma}{dr_i}(r_i)\frac{dh_{i,j}}{d\theta_i}(\theta_i)dr_i\land d\theta_i\right)\left(\frac{\partial}{\partial u}, \frac{\partial}{\partial v}\right)\geq 0
\]
on $F\cap \check W_{i,j}$ by Lemma~\ref{lemmaD}~(3).

Thus $\alpha_R\land d\alpha_R\left(R_{\alpha_R}, \frac{\partial}{\partial u}, \frac{\partial}{\partial v}\right)>0$ holds on the whole $N_D$.
Since $\alpha_R$ is a positive contact form, $R_\alpha$ is positively transverse to
$\left\langle\frac{\partial}{\partial u},\frac{\partial}{\partial v}\right\rangle$
on $N_D$.
This completes the proof.
\end{proof}

\section{Uniqueness}

In this section, we prove the following theorem, which is the uniqueness asserted in Theorem~\ref{thm01}.

\begin{theorem}\label{thm2}
Let $P$ be a positive flow-spine of $M$
and $\alpha_0$ and $\alpha_1$ be contact forms on $M$ whose 
Reeb flows are carried by $P$.
Then the contact structures $\ker\alpha_0$ and $\ker\alpha_1$ are isotopic.
\end{theorem}

\begin{lemma}\label{lemma53}
Let $\mathcal C$ be a non-degenerate coordinate system $\mathcal C$ of $M$ with respect to a positive flow-spine $P$.
Let $\eta$ be the reference $1$-form of $(P,\mathcal C)$.
For $i=0,1$, let $\alpha_i$ be a contact form on $M$ satisfying 
$\alpha_i\land d\eta+d\alpha_i\land\eta>0$. 
Then there exists a one-parameter family of contact forms connecting $\alpha_0$ and $\alpha_1$.
\end{lemma}

\begin{proof}
For each $\alpha_i$, set $\alpha_{R,i}=\alpha_i+R\eta$ for $R\geq 0$. Then
\[
   \alpha_{R,i}\land d\alpha_{R,i}=\alpha_i\land d\alpha_i+R(\alpha_i\land d\eta+d\alpha_i\land\eta)+R^2\eta\land d\eta.
\]
The assumption $\alpha_i\land d\eta+d\alpha_i\land\eta>0$ and 
the inequality $\eta\land d\eta\geq 0$ in Lemma~\ref{lemma35} imply that 
$\alpha_{R,i}$ is a contact form for any $R\geq 0$.

Set $\alpha_{R,s}=(1-s)\alpha_0+s\alpha_1+R\eta$ for $s\in[0,1]$. Then
\[
\begin{split}
\alpha_{R,s}\land d\alpha_{R,s}=&(1-s)^2\alpha_0\land d\alpha_0
+s(1-s)(\alpha_0\land d\alpha_1+\alpha_1\land d\alpha_0)
+s^2\alpha_1\land d\alpha_1 \\
&+R\{((1-s)d\alpha_0+sd\alpha_1)\land \eta+((1-s)\alpha_0+s\alpha_1)\land d\eta\} \\
&+R^2\eta\land d\eta.
\end{split}
\]
Since the fourth term (having the coefficient $R$) is positive by the assumption 
and the last term is non-negative by Lemma~\ref{lemma35},
this $3$-form is positive for any $s\in [0,1]$ if $R$ is sufficiently large.
Combining this with the observation in the first paragraph, we have the assertion.
\end{proof}

\begin{lemma}\label{lemma54}
Let $\mathcal C$ be a non-degenerate coordinate system of $M$ with respect to a positive flow-spine $P$ and let $\eta$ be the reference $1$-form of $(P,\mathcal C)$.
Then there exists a contact form $\alpha$ 
such that $\alpha\land d\eta+d\alpha\land\eta>0$ and 
its Reeb flow is carried by $P$.
\end{lemma}

\begin{proof}
Let $\alpha$ be a contact form whose Reeb flow is carried by $P$ given in Theorem~\ref{thm1}. 
On $N_P$, it is given, in Lemma~\ref{lemma43}, as $\alpha=\hat\beta+R\eta$.
Thus we have
\[
   \alpha\land d\eta+d\alpha\land \eta=
   \hat\beta\land d\eta+d\hat\beta\land\eta+2R\eta\land d\eta>0
\]
by Lemmas~\ref{lemma42a} and~\ref{lemma35}. On $N_D$,
$\alpha$ is given, in~\eqref{eq42}, as $\alpha_R=(1-\tau(w))\hat\beta_E+\tau(w)\hat\beta_F+R\eta$. Thus we have
\[
\begin{split}
\alpha\land d\eta&+d\alpha\land \eta
=\frac{d\tau}{dw}(w)(-\hat\beta_E+\hat\beta_F)\land\eta\land dw \\
&+((1-\tau(w))\hat\beta_E+\tau(w)\hat\beta_F)\land d\eta 
+((1-\tau(w))d\hat\beta_E+\tau(w)d\hat\beta_F)\land\eta \\
&+2R\eta\land d\eta.
\end{split}
\]
This is positive on $N_A$ by the condition~(b) in Lemma~\ref{lemma41} and Lemma~\ref{lemma35}.
On $N_D\setminus N_A$, 
we have $\alpha\land d\eta=((1-\tau(w))d\hat\beta_E+\tau(w)d\hat\beta_F)\land dw$
since $\eta_D=dw$ by Lemma~\ref{lemmaC}~(1), which is positive by the condition~(a) in Lemma~\ref{lemma41}.
This completes the proof.
\end{proof}


Next, we explain how to choose a coordinate system for a given contact form on $M$
and a positive flow-spine $P$ that carries its Reeb flow.
For each vertex $v_j$ of $P$, we choose a neighborhood $\Nbd(v_j;M)$ of $v_j$ in $M$.
Let $\pr$ be the projection from $\Nbd(v_j;M)$ to $\Real^2$ that identifies points connected by an orbit of the Reeb flow in  $\Nbd(v_j;M)$.
Choose the coordinates $x_j^v$ and $y_j^v$ on $\pr(S(P))$ as shown in Figure~\ref{fig6-1}.

\begin{figure}[htbp]
\begin{center}
\includegraphics[width=5cm, bb=220 587 374 713]{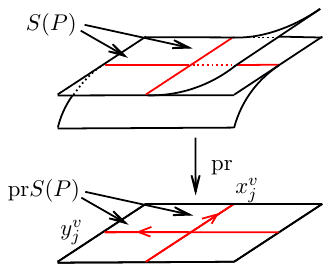}
\caption{The coordinates $x_j^v$ and $y_j^v$ on $\pr(S(P))$.}\label{fig6-1}
\end{center}
\end{figure}

In $\Nbd(v_j;M)$, set the box $[0,3]^2\times [0,5]$ with coordinates $(x_j^v, y_j^v, t_j^v)$, 
that is used in Section~\ref{sec52} to define $B_j^v$, so that 
\begin{itemize}
\item[(i)] the Reeb vector field $R_\alpha$ is parallel to $\frac{\partial}{\partial t_j^v}$ in the same direction, that is, there exists a positive smooth function $c:[0,3]^2\times [0,5]\to \Real$ such that $\frac{\partial}{\partial t_j^v}=c(x_j^v,y_j^v,t_j^v) R_\alpha$ in $[0,3]^2\times [0,5]$, and
\item[(ii)] $P\cap B_j^v$ lies in the core of $B_j^v$.
\end{itemize}

%

The rest pieces $B_j^e$ and $N_{\mathcal R_i}$ of $N(P)$ are chosen canonically. We then choose the coordinates $(u,v,w)$ on $N_D$ so that $\frac{\partial}{\partial w}=C(u,v)R_\alpha$ on $(u,v,w)$, where $C(u,v)$ is a positive smooth function on $(u,v)$. 


We further assume, by choosing $B_j^v$ and $B_j^e$ sufficiently small, that 
\begin{itemize}
\item[(iii)] the coordinate system of $M$ obtained as above is non-degenerate, and
\item[(iv)] the partial derivatives of $C(u,v)$ are sufficiently close to $0$ on $\check W$.
\end{itemize}
The condition~(iv) is satisfied since as $B_j^v$ gets smaller the lengths of Reeb orbits in $\check W_{i,j}$ get close to the same length.


\begin{lemma}\label{lemma52}
Let $P$ be a positive flow-spine of $M$
and $\alpha$ be a contact form on $M$ whose Reeb flow is carried by $P$.
Choose a non-degenerate coordinate system $\mathcal C$ of $M$ with respect to $P$ as above. Assume that, for each $B_j^v$, the coordinate $x_j^v$ satisfies 
$\alpha\left(\frac{\partial}{\partial x_j^v}\right)<0$ in the box $[0,3]^2\times [0,5]$ that contains $B_j^v$.
Then the reference $1$-form $\eta$ of $(P,\mathcal C)$ satisfies $\alpha\land d\eta+d\alpha\land \eta>0$.
\end{lemma}

\begin{proof}
%
On $B_j^k\setminus \bigcup_i g_{\mathcal R_i}(N_{\mathcal R_i})$ for $k=e$ or $v$, we have 
$d\eta=d(dt_j^k)=0$ and $d\alpha\land\eta>0$ since, with the coordinates $(x_j^k,y_j^k,t_j^k)$ of $B_j^k$, 
\[
   d\alpha\land \eta\left(\frac{\partial}{\partial x_j^k}, \frac{\partial}{\partial y_j^k}, R_\alpha\right)
   =dt_j^k(R_\alpha)d\alpha\left(\frac{\partial}{\partial x_j^k}, \frac{\partial}{\partial y_j^k}\right).
\]
It satisfies $dt_j^k(R_\alpha)=dt_j^k\left(
\frac{1}{c}
\frac{\partial}{\partial t_j^k}\right)=
\frac{1}{c}
>0$, where $c=c(x_j^k, y_j^k, t_j^k)$.
We also have
$d\alpha\left(\frac{\partial}{\partial x_j^k}, \frac{\partial}{\partial y_j^k}\right)>0$ otherwise
$\alpha$ cannot be a positive contact form.
Thus $\alpha\land d\eta+d\alpha\land\eta>0$ holds on $B_j^k\setminus \bigcup_i g_{\mathcal R_i}(N_{\mathcal R_i})$.

On $N_{\mathcal R_i}\setminus \bigcup_jW_{i,j}$, we have $d\eta=0$ by Lemma~\ref{lemma31}.
The inequality $d\alpha\land\eta>0$ follows from Lemma~\ref{lemmaB} and the same observation used in the last paragraph.
Thus $\alpha\land d\eta+d\alpha\land\eta>0$ holds on $N_{\mathcal R_i}\setminus \bigcup_jW_{i,j}$.

On $W_{i,j}$, by Lemmas~\ref{lemmaB} and~\ref{lemma31},
\[
\alpha\land d\eta+d\alpha\land\eta=\frac{d\sigma}{dr_i}(r_i)\frac{dh_{i,j}}{d\theta_i}(\theta_i)\alpha\land dr_i\land d\theta_i
+d\alpha\land \left(dt_i+\sigma(r_i)\frac{dh_{i,j}}{d\theta_i}(\theta_i)d\theta_i\right).
\]
The factors $\frac{d\sigma}{dr_i}(r_i)$ and $\frac{dh_{i,j}}{d\theta_i}(\theta_i)$ are non-negative and the $3$-form $\alpha\land dr_i\land d\theta_i$ is positive 
since 
$\frac{\partial}{\partial t_i}=cR_\alpha$ implies
$dr_i\land d\theta_i(R_\alpha,\cdot)=0$.
Thus the first term is non-negative.
We have $dt_i(R_\alpha)>0$ and $d\theta_i(R_\alpha)=0$,
and also have
$d\alpha\left(\frac{\partial}{\partial r_i},\frac{\partial}{\partial \theta_i}\right)>0$
otherwise $\alpha$ cannot be a positive contact form.
Hence the second term is positive. 
Thus $\alpha\land d\eta+d\alpha\land\eta>0$ holds on $W_{i,j}$.

On $N_D$, 
\[
   d\alpha\land\eta\left(\frac{\partial}{\partial u},\frac{\partial}{\partial v},R_\alpha \right)
   =\eta(R_\alpha)d\alpha\left(\frac{\partial}{\partial u},\frac{\partial}{\partial v} \right).
\]
Since $\frac{\partial}{\partial w}=C(u,v)R_\alpha$, we have $\eta(R_\alpha)>0$ 
on $N_D\setminus N_A$ by Lemma~\ref{lemmaC}~(1).
The inequality $\eta(R_\alpha)>0$ also holds on $N_A\setminus (\hat W\cup \check W)$
since $\eta$ is a linear sum of some of $dt_j^k$ for $k=e, v$ and $dt_i$
and on $\hat W\cup \check W$ by Lemmas~\ref{lemmaA} and~\ref{lemmaD}.
We also have $d\alpha\left(\frac{\partial}{\partial u},\frac{\partial}{\partial v} \right)>0$ 
otherwise $\alpha$ cannot be a contact form. Thus $d\alpha\land\eta>0$.
It remains to show the inequality $\alpha\land d\eta\geq 0$ on $N_D$. 
%
On $N_D\setminus (\hat W\cup \check W)$, we have $d\eta=0$ by Lemma~\ref{lemmaC}~(1) and~(2).

We check $\alpha\land d\eta\geq 0$ on $\hat W_{i,j}$.
From~\eqref{eq33}, the $2$-form $d\eta$ on $N_D$ is given as
\[
d\eta=-\frac{d\tau}{dw}(w)dw\land \eta_E
+\frac{d\tau}{dw}(w)dw\land \eta_F
+(1-\tau(w))d\eta_E+\tau(w)d\eta_F.
\]
By Lemma~\ref{lemmaA}~(1) and~(2), 
$dw\land \eta_E$ is the extension of 
$g_E^*\left(\sigma(r_i)\frac{dh_{i,j}}{d\theta_i}(\theta_i)dt_i\land d\theta_i\right)$
and $d\eta_E$ is the extension of 
$g_E^*\left(\frac{d\sigma}{dr_i}(r_i)\frac{dh_{i,j}}{d\theta_i}(\theta_i)dr_i\land d\theta_i\right)$
to the whole $\hat W_{i,j}$.
Similarly, by Lemma~\ref{lemmaA}~(3),
$dw\land \eta_F$ is the extension of 
$g_E^*\left(\frac{dh_{i,j}}{d\theta_i}(\theta_i)dt_i\land d\theta_i\right)$
to $\hat W_{i,j}$ and $d\eta_F=0$ on $\hat W_{i,j}$.
Hence
\[
d\eta=g_E^*\left((1-\sigma(r_i))\frac{d\tau}{dw}(w)\frac{dh_{i,j}}{d\theta_i}(\theta_i)dt_i\land d\theta_i+(1-\tau(w))\frac{d\sigma}{dr_i}(r_i)\frac{dh_{i,j}}{d\theta_i}(\theta_i)dr_i\land d\theta_i\right).
\]
On $E\cap\hat W_{i,j}$, we have
\[
\begin{split}
\alpha\land d\eta\left(\frac{\partial}{\partial r_i},\frac{\partial}{\partial \theta_i},
R_\alpha \right)
=&
-(1-\sigma(r_i))
\frac{d\tau}{dw}(w)\frac{dh_{i,j}}{d\theta_i}(\theta_i)\alpha\left(\frac{\partial}{\partial r_i}\right)dt_i(R_\alpha) \\
& +  
(1-\tau(w))\frac{d\sigma}{dr_i}(r_i)\frac{dh_{i,j}}{d\theta_i}(\theta_i),
\end{split}
\]
where
the coordinates $(r_i,\theta_i,w)$ are regarded as those on $\hat W_{i,j}$ 
according to the product structure of $\hat W_{i,j}$ in $N_D=D^2\times [-\varepsilon, 1+\varepsilon]$.
Since $\hat W_{i,j}\subset [0,3]^2\times [0,5]$, the inequality $\alpha\left(\frac{\partial}{\partial x_j^v}\right)<0$ holds on $\hat W_{i,j}$ by the assumption, which implies $\alpha\left(\frac{\partial}{\partial r_i}\right)<0$. 
Hence we have $\alpha\land d\eta\left(\frac{\partial}{\partial r_i},\frac{\partial}{\partial \theta_i}, R_\alpha \right)\geq 0$.

We check $\alpha\land d\eta\geq 0$ on $\check W_{i,j}$.
By Lemma~\ref{lemmaD},
\[
d\eta=\frac{d\tau}{dw}(w)g^*_F\left(\sigma(r_i)\frac{dh_{i,j}}{d\theta_i}(\theta_i)dt_i\land d\theta_i\right)
+\tau(w)g_F^*\left(\frac{d\sigma}{dr_i}(r_i)\frac{dh_{i,j}}{d\theta_i}(\theta_i)dr_i\land d\theta_i\right).
\]
On $F\cap\check W_{i,j}$, we have
\[
\begin{split}
\alpha\land d\eta\left(\frac{\partial}{\partial r_i},\frac{\partial}{\partial \theta_i},
R_\alpha \right)
=
&-\frac{d\tau}{dw}(w)\sigma(r_i)\frac{dh_{i,j}}{d\theta_i}(\theta_i)\alpha\left(\frac{\partial}{\partial r_i}\right)dt_i(R_\alpha) \\
& + 
\tau(w)\frac{d\sigma}{dr_i}(r_i)\frac{dh_{i,j}}{d\theta_i}(\theta_i),
\end{split}
\]
where $(r_i, \theta_i, t_i)$ are the coordinates of $W_{i,j}\subset N'_{\mathcal R_i}$ and 
$(r_i,\theta_i,w)$ are regarded as the coordinates on $\check W_{i,j}$ 
according to the canonical product structure of $\check W_{i,j}$ in $N_D=D^2\times [-\varepsilon, 1+\varepsilon]$, see Figure~\ref{fig6-3}.
Since $g_{\mathcal R_i}(W_{i,j})\subset B_j^v$, 
$\alpha\left(\frac{\partial}{\partial r_i}\right)<0$ holds on $W_{i,j}$ by the assumption.
Moreover, the same inequality holds on the whole $\check W_{i,j}$
since $\frac{\partial}{\partial w}=C(u,v)R_\alpha$ satisfies
\[
   \mathcal L_{\frac{\partial}{\partial w}}\alpha=d(\imath_{\frac{\partial}{\partial w}}\alpha)=dC(u,v),
\]
which is sufficiently close to $0$ by the assumption~(iv),
where $\mathcal L_X$ is the Lie derivative and $\imath_X$ 
is the interior product for a vector field $X$.
Hence we have $\alpha\land d\eta\left(\frac{\partial}{\partial r_i},\frac{\partial}{\partial \theta_i},R_\alpha \right)\geq 0$. 

\begin{figure}[htbp]
\begin{center}
\includegraphics[width=12cm, bb=130 532 475 711]{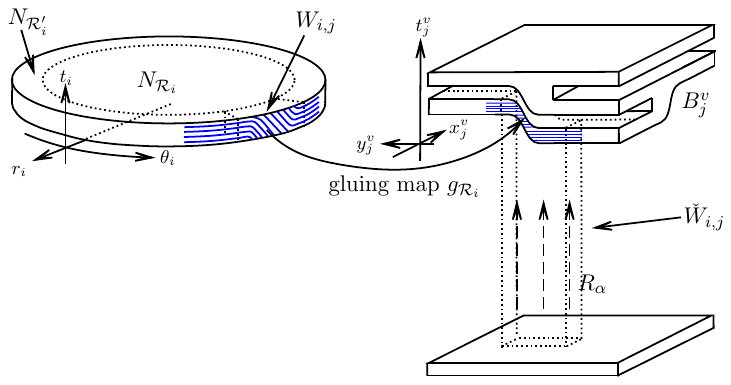}
\caption{The coordinates $(r_i, \theta_i, w)$ on $\check W_{i,j}$.}\label{fig6-3}
\end{center}
\end{figure}

Thus $\alpha\land d\eta+d\alpha\land\eta>0$ is satisfied on the whole $N_D$.
This completes the proof.
\end{proof}


\begin{lemma}\label{lemma55}
For $i=0,1$, 
let $\mathcal C_i$ be a non-degenerate coordinate system of $M$ with respect to a positive flow-spine $P$ and $\eta_i$ be the reference $1$-form of $(P,\mathcal C_i)$.
Let
$\alpha_i$ be a contact form on $M$ such that $\alpha_i\land d\eta_i+d\alpha_i\land\eta_i>0$
and 
its Reeb flow is carried by $P$.
Then there exists a one-parameter family of contact forms 
connecting $\alpha_0$ and $\alpha_1$.
\end{lemma}

\begin{proof}
For each $i=0,1$, let $G_i$ be the graph obtained as the union of $\pr_{D^2}S_F$ and $\pr_{D^2}S_E$ on the unit disk, where the interior points of the unit disk correspond to the Reeb orbits starting and ending at points on $P$.
Choose a deformation $G_t$ of graph from $G_0$ to $G_1$ in a generic way.
Let $t_1,\ldots,t_s$ be the points in $[0,1]$ at which $G_{t_j}$ has a self-tangency or a triple point or an edge passing through a trivalent vertex. We call them degenerate times.
For each $t_j$, we may obtain a non-degenerate coordinate system $\mathcal C_{t_j}$ from the graph $G_{t_j}$ canonically. 
Let $P_{t_j}$ denote the flow-spine $P$ in the coordinate system $\mathcal C_{t_j}$
and $\eta_{t_j}$ be the reference $1$-form of $(P_{t_j}, \mathcal C_{t_j})$.
We then make a contact form $\alpha_{t_j}$ on $M$ as in Theorem~\ref{thm1}. 

Next we perturb the coordinate system $\mathcal C_{t_j}$ slightly with fixing the contact form $\alpha_{t_j}$ so that the degenerate time of $G_{t_j}$ is resolved. 
Figure~\ref{fig8-1} shows an example in the case when an edge passes through a trivalent vertex. The figure in the middle is $G_{t_j}$. By small perturbations of 
$\mathcal C_{t_j}$, 
we may obtain the graphs on the left and on the right in the figure.
One of the perturbed graphs can be regarded as $G_{t_j-\varepsilon}$ and the other can be regarded as $G_{t_j+\varepsilon}$.
For each of them, we choose $A_{R_i}$ sufficiently narrow so that it satisfies the condition (A) in Section~\ref{sec61}. 

\begin{figure}[htbp]
\begin{center}
\includegraphics[width=13.5cm, bb=129 604 483 713]{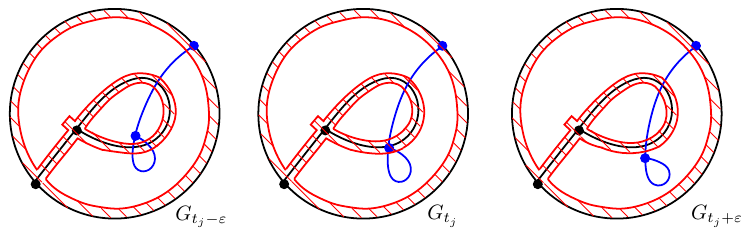}
\caption{Non-degenerate coordinate systems before and after passing through a trivalent vertex.}
\label{fig8-1}
\end{center}
\end{figure}

Let $P_{t_j-\varepsilon}$ and $P_{t_j+\varepsilon}$ denote the flow-spine $P_{t_j}$ in the coordinate systems $C_{t_j-\varepsilon}$ and $C_{t_j+\varepsilon}$, respectively.
We isotope $P_{t_j-\varepsilon}$ and $P_{t_j+\varepsilon}$
in a neighborhood of each vertex so that it satisfies the assumption in Lemma~\ref{lemma52}
if necessary.
This is done as shown in Figure~\ref{fig8-2}.
The left in the figure is the case where we have $\alpha(\frac{\partial}{\partial x_j^v})>0$, that is, the assumption is not satisfied.
By rotating the flow-spine near the vertex  as shown on the right, we may set the coordinates $(x_j^v, y_j^v, t_j^v)$ so that $\alpha(\frac{\partial}{\partial x_j^v})<0$.

\begin{figure}[htbp]
\begin{center}
\includegraphics[width=12cm, bb=137 583 434 712]{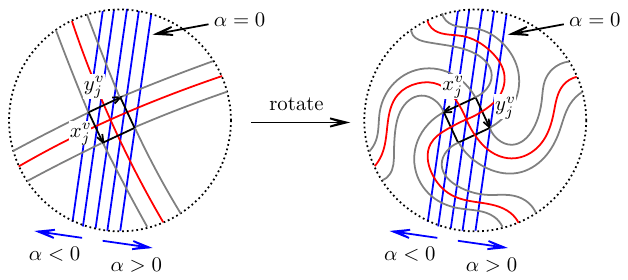}
\caption{Rotate the flow-spine so that the coordinates $(x_j^v, y_j^v, t_j^v)$ satisfy 
$\alpha(\frac{\partial}{\partial x_j^v})<0$.}
\label{fig8-2}
\end{center}
\end{figure}

Applying Lemma~\ref{lemma52} to these isotoped $P_{t_j-\varepsilon}$ and $P_{t_j+\varepsilon}$,
we obtain non-degenerate coordinate systems $\mathcal C_{t_j-\varepsilon}$ and $\mathcal C_{t_j+\varepsilon}$ for  $G_{t_j-\varepsilon}$ and $G_{t_j+\varepsilon}$ whose reference $1$-forms 
 $\eta_{t_j-\varepsilon}$ and $\eta_{t_j+\varepsilon}$ satisfy
$\alpha_{t_j}\land d\eta_{t_j-\varepsilon}+d\alpha_{t_j}\land \eta_{t_j-\varepsilon}>0$ 
and $\alpha_{t_j}\land d\eta_{t_j+\varepsilon}+d\alpha_{t_j}\land \eta_{t_j+\varepsilon}>0$,
respectively. 
We set $\alpha_t=\alpha_{t_j}$ for $t\in [t_j-\varepsilon, t_j+\varepsilon]$.
Set $t_0=0$, $t_{s+1}=1$, $\alpha_t=\alpha_0$ for $t\in [0,\varepsilon]$
and $\alpha_t=\alpha_1$ for $t\in [1-\varepsilon, 1]$.

Now, for each $j=0, \ldots, s$, we choose a smooth family 
$\{\mathcal C_t \mid t\in [t_j+\varepsilon, {t_{j+1}}-\varepsilon]\}$ 
of non-degenerate coordinate systems 
 connecting $\mathcal C_{t_j+\varepsilon}$ and $\mathcal C_{t_{j+1}-\varepsilon}$.
Then a smooth family of reference $1$-form $\eta_t$ is also obtained for $t\in [t_j+\varepsilon, {t_{j+1}}-\varepsilon]$.
For each $t\in [t_j+\varepsilon, {t_{j+1}}-\varepsilon]$, the set of contact forms $\alpha$ 
satisfying $\alpha\land d\eta_t+d\alpha\land\eta_t>0$ is open with Whitney $C^\infty$-topology,
connected by Lemma~\ref{lemma53} and non-empty by Lemma~\ref{lemma54}. 
Therefore there exists a one-parameter family of contact forms connecting $\alpha_{t_j+\varepsilon}$ and $\alpha_{t_{j+1}-\varepsilon}$ that 
satisfies $\alpha_t\land d\eta_t+d\alpha_t\land\eta_t>0$ for $t\in [t_j+\varepsilon, {t_{j+1}}-\varepsilon]$.
Since 
$\alpha_t=\alpha_{t_j}$ for $t\in [t_j-\varepsilon, t_j+\varepsilon]$,
the family $\alpha_t$, $t\in [0,1]$ is a one-parameter family of contact forms connecting $\alpha_0$ and $\alpha_1$.
\end{proof}

\begin{proof}[Proof of Theorem~\ref{thm2}]
For each $\alpha_i$, $i=0,1$,
we isotope $P_i$ in a neighborhood of each vertex as shown in Figure~\ref{fig8-2},
if necessary, 
so that it satisfies the assumption in Lemma~\ref{lemma52}.
Applying Lemma~\ref{lemma52}, we obtain a non-degenerate coordinate system $\mathcal C_i$ such that the reference $1$-form $\eta_i$ of $(P_i,\mathcal C_i)$ satisfies $\alpha_i\land d\eta_i+d\alpha_i\land \eta_i>0$.
Thus the contact forms $\alpha_0$ and $\alpha_1$ are connected by a one-parameter family of contact forms by Lemma~\ref{lemma55}.
By the Gray stability~\cite{Gra59}, $\ker\alpha_0$ and $\ker\alpha_1$ are isotopic.
\end{proof}

\end{document}